\documentclass[12pt]{amsart}
\usepackage{amsmath,amssymb,txfonts}
\usepackage{amssymb}
\usepackage{amsmath}
\usepackage{mathrsfs}
\usepackage{amsmath,amssymb}

\usepackage{color}

\newtheorem{theorem}{Theorem}[section]
\newtheorem{lemma}[theorem]{Lemma}

\newtheorem{corollary}[theorem]{Corollary}
\theoremstyle{definition}
\newtheorem{definition}[theorem]{Definition}

\theoremstyle{remark}
\newtheorem{remark}[theorem]{Remark}

\numberwithin{equation}{section}

\begin{document}

\title[ Triebel type mean oscillation spaces]
{Wavelets and Triebel type oscillation spaces
}

\author[Pengtao Li]{Pengtao Li}
\address{Department of Mathematics, Shantou University, Shantou, Guangdong, China.}
\email{ptli@stu.edu.cn}

\author[Qixiang Yang]{Qixiang Yang}\thanks{Qixiang Yang is the corresponding author}
\address{School of Mathematics and Statics, Wuhan University, Wuhan, 430072, China.}
\email{qxyang@whu.edu.cn}
\author[Bentuo Zheng]{Bentuo Zheng}
\address{Department of Mathematical Sciences, University of Memphis, Memphis, TN 38152-3240.}
\curraddr{}
\email{bzheng@memphis.edu}
\dedicatory{}

\thanks{Pengtao Li's research is supported by NSFC No.11171203,  11201280;
New Teacher's Fund for Doctor Stations, Ministry of Education No.20114402120003;
Guangdong Natural Science Foundation S2011040004131; Foundation for
Distinguished Young Talents in Higher Education of Guangdong, China,
LYM11063. Qixiang Yang's research is supported in part by NSFC No. 11271209.
Bentuo Zheng's research is supported in part by NSF grant DMS-1200370.}

\maketitle

\date{}

\vspace{0.5cm}




\thanks{}
\thanks{}


\date{}

\dedicatory{}

\begin{abstract}
We apply wavelets to identify the Triebel type oscillation
spaces with the known Triebel-Lizorkin-Morrey spaces
$\dot{F}^{\gamma_1,\gamma_2}_{p,q}(\mathbb{R}^{n})$.
Then 
we establish a  characterization of
$\dot{F}^{\gamma_1,\gamma_2}_{p,q}(\mathbb{R}^{n})$ via the fractional heat semigroup.
Moreover, we prove the continuity of Calder\'on-Zygmund operators on these spaces. The results of
this paper also provide necessary tools for the study of
well-posedness of Navier-Stokes equations.
\end{abstract}

\maketitle \tableofcontents \pagenumbering{arabic}
\section{Introduction}\label{sec1}
We  state briefly the history of Triebel-Lizorkin spaces and their Morrey type generalization.
Triebel-Lizorkin spaces $\dot{F}^{s}_{p,q}(\mathbb{R}^{n})$ were first introduced by H. Triebel
and can be seen as generalizations of many standard function spaces
such as Lebesgue spaces $L^{p}$ and Sobolev spaces.
In the research of harmonic analysis and partial differential equations,
Triebel-Lizorkin spaces play an important role. In recent decades, $\dot{F}^{s}_{p,q}(\mathbb{R}^{n})$ have attracted great attention of many mathematicians, and  a lot of work has been done. We refer the readers to Triebel \cite{Triebel-1, Triebel-2} for an overview of Triebel-Lizorkin spaces and their applications.

D. Yang and his collaborators are pioneers on the study of Triebel type Morrey spaces. By applying Hausdorff capacity and Littlewood-Paley theory,
Yang-Yuan \cite{YY} introduced a new class of function spaces $\dot{F}^{s, \tau}_{p,q}(\mathbb{R}^{n})$
with $p\in(0, \infty)$ which generalize  many classical function spaces. For example, $\dot{F}^{s, \frac{n}{p}}_{p,q}(\mathbb{R}^{n})=\dot{F}^{s}_{p,q}(\mathbb{R}^{n})$. $\dot{F}^{\alpha, 1/2-\alpha/n}_{2,2}(\mathbb{R}^{n})=Q_{\alpha}(\mathbb{R}^{n})$, where $Q_{\alpha}(\mathbb{R}^{n})$ are the spaces introduced  by Ess\'en-Janson-Peng-Xiao \cite{EJPX}. For more information, we refer to Yuan-Sickel-Yang \cite{YSY}.

Our aim is to study a class of mean oscillation spaces with Triebel-Lizorkin norm by wavelets and semigroup.
In this paper, the Triebel type oscillation spaces
$\dot{F}^{\gamma_{1},\gamma_{2}}_{p,q}(\mathbb{R}^{n})$ are defined as
$$
\sup\limits_{ Q}
|Q|^{\frac{\gamma_{2}}{n}-\frac{1}{p}} \inf\limits_{P_{Q, f}\in
S^{\gamma_{1},\gamma_{2}}_{p,q,f}} \|\varphi_{Q} (f-P_{Q,
f})\|_{\dot{F}^{\gamma_{1}, q}_{p}} <+ \infty,
$$
where the supremum is taken over all cubes $Q$ and $S^{\gamma_{1},\gamma_{2}}_{p,q,f}$ denotes the set of all polynomials satisfying certain conditions. Details can be found in Definition \ref{def1}. In Theorem \ref{lem:c}, we give a wavelet  characterization of these spaces.
As a consequence, $\dot{F}^{\gamma_{1},\gamma_{2}}_{p,q}(\mathbb{R}^{n})$
coincide with $\dot{F}^{s,\tau}_{p,q}(\mathbb{R}^{n})$ introduced by  Yang-Yuan \cite{YY}.
Theorem \ref{lem:c} implies that Calder\'on -Zygmund operators are bounded on $\dot{F}^{\gamma_{1},\gamma_{2}}_{p,q}(\mathbb{R}^{n})$. Moreover, our wavelet characterization is independent of the choice of wavelet bases. See Corollary \ref{coro-Banach} and Theorem \ref{lem:CZcontinuity}, respectively.

It is well-known that for any $f\in BMO(\mathbb{R}^{n})$, the Poisson integral $P_{t}(f)$
gives a harmonic extension of $f$ to the tent space on $\mathbb{R}^{n+1}_{+}$.
This result gives a relation between function spaces on $\mathbb{R}^{n}$ and the ones on $\mathbb{R}^{n+1}_{+}$.
The well-posedness of fluid equations needs often such characterizations of function spaces.
In 2001,  Koch-Tataru \cite{KT} obtained a semigroup characterizations of $BMO(\mathbb{R}^{n})$.
In 2007, by Hausdorff capacity, Xiao \cite{X1} gave a semigroup characterization of $Q_{\alpha}(\mathbb{R}^{n})$.
 Li-Zhai \cite{LZ1} further developed the idea of \cite{KT, X1}  and obtained a semigroup characterization of $Q^{\beta}_{\alpha}(\mathbb{R}^{n})$.
We  refer the readers to Cannone \cite{C, C2}, Li-Xiao-Yang \cite{LXY}, Lin-Yang \cite{LinYang} and Miao-Yuan-Zhang \cite{C. Miao B. Yuan  B. Zhang} for further information.

In Section \ref{sec3}, we introduce tent type spaces $\mathbb{F}^{\gamma_{1}, \gamma_{2}}_{p,q,m,m'}$
and $\mathbb{F}^{\gamma}_{\tau, \infty}$ defined on $\mathbb{R}^{n+1}_{+}$
and study some properties of these spaces. In Section \ref{sec:semigroup}, via fractional heat semigroup,
 we establish a relation between the functions in $\dot{F}^{\gamma_{1},\gamma_{2}}_{p,q}(\mathbb{R}^{n})$ and $\mathbb{F}^{\gamma_{1}, \gamma_{2}}_{p,q,m,m'}$.

$\mathbb{F}^{\gamma_{1}, \gamma_{2}}_{p,q,m,m'}$ is defined as follows:
$$\begin{array}{rl}
\mathbb{F}^{\gamma_{1}, \gamma_{2}}_{p,q,m,m'}&=\mathbb{F}^{\gamma_{1}, \gamma_{2}, I}_{p,q,m }\cap \mathbb{F}^{\gamma_{1}, \gamma_{2},II}_{p,q}
\cap\mathbb{F}^{\gamma_{1}, \gamma_{2}, III}_{p,q,m}\cap \mathbb{F}^{\gamma_{1}, \gamma_{2}, IV}_{p,q, m'}\\
&=:X_{1}\cap X_{2}\cap X_{2}\cap X_{4}.
\end{array}$$

 Actually, Theorem \ref{thm-cha} is not a simple generalization of the results in \cite{KT, LZ1, X1}. In the above mentioned spaces, $BMO$, $Q_{\alpha}$ and $Q^{\beta}_{\alpha}$ are all $\dot{F}^{\gamma_{1},\gamma_{2}}_{p,q}(\mathbb{R}^{n})$ spaces with $p=q=2$. For the cases $p\neq q$ with $p, q\neq2$, the Fourier transform is not valid. To overcome this difficulty, we apply a new method. Let $\{\Phi^{\varepsilon}_{j,k}(x)\}_{(\varepsilon,j,k)\in\Lambda_{n}}$ be a wavelet basis. Let $Q$ be any cube and
$$\begin{array}{rl}
f(x)=\sum\limits_{(\varepsilon,j,k)\in\Lambda_{n}}a^{\varepsilon}_{j,k}\Phi^{\varepsilon}_{j,k}(x)\in\dot{F}^{\gamma_{1},\gamma_{2}}_{p,q}(\mathbb{R}^{n}).
\end{array}$$
Based on the relation between $j$ and the radius of $Q$,
we decompose the function $F(x,t)=:e^{t(-\Delta)^{\beta}}f(x)$ into several parts such that
every part belongs to some $X_{i}$.
Such decomposition reflects the local structures of the space and  the frequency very well.
A semigroup characterization of $\dot{F}^{\gamma_{1}, \gamma_{2}}_{p,q}(\mathbb{R}^{n})$ can be obtained easily.

Our  characterization has a distinct advantage  when  
 we apply it  to the well-posedness of fluid equations.  Roughly speaking,
for $F(x, t)\in\mathbb{F}^{\gamma_{1},\gamma_{2}}_{p,q,m,m'}$,
the four parts of $\|F\|_{\mathbb{F}^{\gamma_{1},\gamma_{2}}_{p,q,m,m'}}$ have different meanings:
\begin{itemize}
\item the norms $\|F\|_{X_{1}}$ and $\|F\|_{X_{2}}$ denote the $L^{\infty}$-parts of $F(x, t)$,
\item the norms $\|F\|_{X_{3}}$ and $\|F\|_{X_{4}}$ denote the $L^{p}-$parts of $F(x, t)$.
\end{itemize}
Furthermore, the index $m$ represents the regularities for the variable $x$.
Compared with the results in \cite{KT, LZ1, X1}, if $m$ becomes bigger,
the elements in $\mathbb{F}^{\gamma_{1}, \gamma_{2}}_{p,q,m,m'}$ have higher regularities.
Moreover, Riesz operators are continuous on $\mathbb{F}^{\gamma_{1}, \gamma_{2}}_{p,q,m,m'}$.
We will also use such characterization to study the well-posedness of  Navier-Stokes equations in another paper.

 The rest of this paper is organized as follows. In Section \ref{sec2}, we
present some preliminary knowledge, notations and terminology.
Then we give a wavelet characterization of
$\dot{F}^{\gamma_{1},\gamma_{2}}_{p,q}(\mathbb{R}^{n})$ and prove
 Calder\'on-Zygmund operators are bounded on $\dot{F}^{\gamma_{1},\gamma_{2}}_{p,q}(\mathbb{R}^{n})$.
In Section \ref{sec3}, we introduce  Triebel type tent spaces.
In the last section, we establish first a relation between $\dot{F}^{\gamma_{1},\gamma_{2}}_{p,q}(\mathbb{R}^{n})$
and  $\mathbb{F}^{\gamma_{1},\gamma_{2}}_{p, q, m, m'}$.
Then, we prove the continuity of Riesz operators on
$\mathbb{F}^{\gamma_{1},\gamma_{2}}_{p, q, m, m'}$.


\section{Triebel type oscillation spaces $\dot{F}^{\gamma_{1}, \gamma_{2}}_{p, q}$}\label{sec2}
In this paper, the symbols $\mathbb{Z} $  and $\mathbb{N}$ denote the
sets of all integers and natural numbers, respectively. For $n\in
\mathbb{N},$ $\mathbb{R}^{n}$ is the $n-$dimensional Euclidean
space, with Euclidean norm denoted by $|x|$ and Lebesgue measure
denoted by $dx$. $\mathbb{R}^{n+1}_{+}$ is the  upper half-space
$\left\{(t,x)\in \mathbb{R}^{n+1}_{+}: t>0, x\in
\mathbb{R}^{n}\right\}$ with Lebesgue measure $dtdx.$ $B(x,r)$
denotes the ball in $\mathbb{R}^{n}$ with center $x$, radius $r$ and
volume $|B|$. Denote by $Q$  a cube in $\mathbb{R}^{n}$ with
sides parallel to the coordinate axes. The volume and side length of
$Q$ are denoted by $|Q|$ and
 $l(Q)$, respectively.

For convenience, the positive constants $C$
may change and usually depend
on the dimension $n,$ $\alpha,$ $\beta$
and other parameters.
The Schwartz class of rapidly decreasing functions and its dual will be denoted by
$\mathscr{S}(\mathbb{R}^{n})$ and $\mathscr{S}'(\mathbb{R}^{n}),$ respectively.
For a  function $f\in \mathscr{S}(\mathbb{R}^{n}),$  $\widehat{f}$
means  the Fourier transform of $f.$

\subsection{Wavelets}
In this paper, we use real valued tensor product orthogonal  wavelets $\Phi^{\epsilon}(x)$
which will be Daubechies wavelets or classical Meyer wavelets.
Daubechies wavelets are only used  in Section \ref{sec2.2} and Meyer wavelets will be used throughout this paper.
If $\Phi^{\epsilon}(x)$ is a Daubechies wavelet, we assume that there exists
a sufficiently big integer $m_0$ which is greater than some constant depending on the  index of the relative Triebel type oscillation spaces such that
\begin{itemize}
\item[(1)] $\forall (\epsilon\in\{0,1\}^n$,
$\Phi^{\epsilon}(x)\in C^{m_0}_{0}([-2^{M},2^{M}]^{n})$;
\item[(2)]
For any $\epsilon\in E_{n}$, $\Phi^{\epsilon}(x)$ has the vanishing moments
up to the order $m_0-1$.
\end{itemize}

We state some preliminaries on classic Meyer wavelets.
Let $\Psi^{0}(\xi)\in C^{\infty}_{0}([-\frac{4\pi}{3},
\frac{4\pi}{3}])$ be an even function satisfying
$$\left\{ \begin{aligned}
&\Psi^{0}(\xi)\in [0,1],\\
&\Psi^{0}(\xi)=1, \text{ if }|\xi|\leq \frac{2\pi}{3}.
\end{aligned}
\right.$$
 Let $\Omega(\xi)= ((\Psi^{0}(\frac{\xi}{2})^{2}-
(\Psi^{0}(\xi))^{2})^{\frac{1}{2}} $. Then $\Omega(\xi)\in
C^{\infty}_{0}([-\frac{8\pi}{3}, \frac{8\pi}{3}])$ is an even
function satisfying:
$$\left\{ \begin{aligned}
&\Omega(\xi)=0, \text{ if }|\xi|\leq \frac{2\pi}{3};\\
&\Omega^{2}(\xi)+\Omega^{2}(2\xi)=\Omega^{2}(\xi)+\Omega^{2}(2\pi-\xi)=1,
\text{ if }\xi\in [\frac{2\pi}{3},\frac{4\pi}{3}].
\end{aligned}
\right.$$
Let $\Psi^{1}(\xi)=
\Omega(\xi) e^{-\frac{i\xi}{2}}$. For any $\epsilon=
(\epsilon_{1},\cdots, \epsilon_{n}) \in \{0,1\}^{n}$, let the
Fourier transform of $\Phi^{\epsilon}(x)$ be
$\hat{\Phi}^{\epsilon}(\xi)= \prod\limits^{n}_{i=1}
\Psi^{\epsilon_{i}}(\xi_{i})$.

For $j\in \mathbb{Z},k\in\mathbb{Z}^{n}$, let $\Phi^{\epsilon}_{j,k}(x)=
2^{\frac{nj}{2}} \Phi^{\epsilon} (2^{j}x-k)$. In this paper, we
denote
$$\begin{array}{rl}
E_{n}:=& \{0,1\}^{n}\backslash\{0\},\\
F_{n}:=& \{(\epsilon,k):\epsilon\in E_{n},
k\in\mathbb{Z}^{n}\},\\
\Lambda_{n}=&\{(\epsilon,j,k), \epsilon\in E_{n}, j\in\mathbb{Z},
k\in \mathbb{Z}^{n}\}.\end{array}$$
For further information about wavelets,
we refer the reader to Meyer \cite{Me}, Wojtaszczyk \cite{Woj} and Yang \cite{Yang1}.
The following result is well-known.
\begin{lemma}\label{wavelet basis}
$\{\Phi^{\epsilon}_{j,k}(x)\}_{(\epsilon,j,k)\in \Lambda_{n}}$ is an orthogonal
basis in $L^{2}(\mathbb{R}^{n})$.
\end{lemma}
For function $f(x)$, $\forall \epsilon\in \{0,1\}^{n}$ and $k\in
\mathbb{Z}^{n}$, denote by $f^{\epsilon}_{j,k}= \langle f(x),
\Phi^{\epsilon}_{j,k}(x)\rangle$ the wavelet coefficients of $f$.
Let
$$P_{j}f(x)= \sum\limits_{k\in \mathbb{Z}^{n}} f^{0}_{j,k}
\Phi^{0}_{j,k}(x)\text{ and }f_{j}(x)=Q_{j}f(x)=
\sum\limits_{(\epsilon,k)\in F_{n}} f^{\epsilon}_{j,k}
\Phi^{\epsilon}_{j,k}(x).$$ By Lemma \ref{wavelet basis}, we can see
that $P_{j}$ and $Q_{j}$ are two projection operators on
$L^{2}(\mathbb{R}^{n})$. In fact, for any two functions $u$ and $v$,
we have
\begin{equation}\label{eqn2.1}
\begin{array}{rl}
uv= & \sum\limits_{j\in \mathbb{Z}} P_{j-3}u Q_{j}v + \sum\limits_{j\in \mathbb{Z}} Q_{j}u Q_{j}v
+ \sum\limits_{0<j-j'\leq 3} Q_{j}u Q_{j'}v \\
&+ \sum\limits_{0<j'-j\leq 3} Q_{j}u Q_{j'}v
+ \sum\limits_{j\in \mathbb{Z}} Q_{j}u P_{j-3}v
\end{array}
\end{equation}

\subsection{Characterization via Daubechies wavelets}\label{sec2.2}
Now we introduce a class of Triebel type oscillation spaces.
Let $\varphi(x)\in C^{\infty}_{0} (B(0,n))$ be such that $\varphi(x)=1 $ for $x\in
B(0, \sqrt{n})$. Let $Q(x_{0},r)$ be a cube with sides parallel to the
coordinate axis, centered at $x_{0}$ and with side length $r$. To
simplify the notation, sometimes, we denote $Q=Q(r)=Q(x_{0},r)$ and
let $\varphi_{Q}(x)= \varphi(\frac{x-x_{Q}}{r})$. For $0<p,q\leq
\infty$ and $\gamma_{1}, \gamma_{2}\in \mathbb{R}$, let
$m_{0}=m^{\gamma_{1},\gamma_{2}}_{p,q}$ be a sufficiently big positive
real number. For arbitrary function $f(x)$, let
$S^{\gamma_{1},\gamma_{2}}_{p,q,f}$ be the set of polynomial
functions $P_{Q, f}(x)$ such that $\forall |\alpha|\leq m_{0}$,
 $$\int x^{\alpha} \varphi_{Q}(x) (f(x) -P_{Q,f}(x)) dx =0.$$

\begin{definition}\label{def1}
Given $0<p<\infty, 0< q\leq \infty$ and $\gamma_{1}, \gamma_{2}\in \mathbb{R}$.
Triebel type oscillation spaces
$\dot{F}^{\gamma_{1},\gamma_{2}}_{p,q}(\mathbb{R}^{n})$ are defined
as:

\begin{equation}\label{eq:b}
\sup\limits_{{\mbox{ cube } Q}}
|Q|^{\frac{\gamma_{2}}{n}-\frac{1}{p}} \inf\limits_{P_{Q, f}\in
S^{\gamma_{1},\gamma_{2}}_{p,q,f}} \|\varphi_{Q} (f-P_{Q,
f})\|_{\dot{F}^{\gamma_{1}, q}_{p}} <+ \infty,
\end{equation}
where the supremum is taken over all the cubes in $\mathbb{R}^{n}$.
\end{definition}

Let $0<p,q\leq \infty$ and $ \gamma_{1}, \gamma_{2}\in \mathbb{R}$.
There exists a sufficiently big integer $m^{\gamma_{1},\gamma_{2}}_{p,q}$
such that regular Daubechies can characterize such spaces.
We call a Daubechies wavelets $\Phi^{\epsilon}(x)$ regular if there exist
two integers $m_0\geq m^{\gamma_{1},\gamma_{2}}_{p,q}$ and $M$
such that
\begin{eqnarray}
&&\forall \epsilon\in \{0,1\}^n, \Phi^{\epsilon}(x)\in
C^{m_0}_0([-2^M, 2^M]^n);\\
&&\forall \epsilon\in E_{n}, \int x^{\alpha} \Phi^{\epsilon}(x)
dx=0,\forall |\alpha|\leq m_{0}.\label{canceling-condition}
\end{eqnarray} Using Daubechies wavelets, we have the following wavelet
characterization of $\dot{F}^{\gamma_{1},\gamma_{2}}_{p,q}$:

\begin{theorem}\label{lem:c}Given $0<p<\infty,0<q\leq \infty$ and $ \gamma_{1}, \gamma_{2}\in \mathbb{R}$.
$f(x)= \sum\limits_{(\epsilon,j,k)\in \Lambda_n} a^{\epsilon}_{j,k}
\Phi^{\epsilon}_{j,k}(x)\in \dot{F}^{\gamma_{1},\gamma_{2}}_{p,q}$
if and only if
\begin{equation}\label{eq:c}
\begin{array}{rl}
&\sup\limits_{ Q} |Q|^{\frac{\gamma_{2}}{n}-\frac{1}{p}}
\Big\|\Big(\sum\limits_{(\epsilon,j,k)\in\Lambda_{Q}^{n}}2^{qj(\gamma_{1}+\frac{n}{2})}|a^{\varepsilon}_{j,k}|^{q}\chi(2^{j}\cdot-k)\Big)^{1/q}
\Big\|_{L^{p}} <+ \infty,
\end{array}
\end{equation}
where the supremum is taken over all the dyadic cubes $Q$ in
$\mathbb{R}^{n}$.
\end{theorem}

\begin{proof} we prove first that
$f(x)\in \dot{F}^{\gamma_{1},\gamma_{2}}_{p,q}$ implies $f(x)$
satisfies (\ref{eq:c}).
For any dyadic cube $Q$ with center $x_{Q}$ and side length $l(Q)$,
there exists a cube $\tilde{Q}$, parallel to the coordinate axis,
centered at $x_{Q}$ and with side length $2^{M+2}l(Q)$.
By definition of $\varphi_{Q}(x)$ and (\ref{canceling-condition}),
for such $\tilde{Q}$ and $ Q_{j,k}\subset Q, x\in\tilde{Q}$,
we have $f(x)= \varphi_{\tilde{Q}}(x) f(x)$ and
$$\int \big(\varphi_{\tilde{Q}}(y) P_{\tilde{Q}, f}(y)\big) \Phi^{\epsilon}_{j,k}(y) dy
=\int P_{\tilde{Q}, f}(y) \Phi^{\epsilon}_{j,k}(y) dy=0.$$
Hence,  for any $\epsilon\in E_{n}$ and $ Q_{j,k}\subset Q$, we have
$$\langle f, \Phi^{\epsilon}_{j,k}\rangle=
\langle \varphi_{\tilde{Q}}(f-P_{\tilde{Q}, f}),
\Phi^{\epsilon}_{j,k}\rangle.$$ By wavelet
characterization of Triebel-Lizorkin spaces, we have
$$\begin{array}{rl}
&\Big\|\Big(\sum\limits_{(\epsilon,j,k)\in\Lambda_{Q}^{n}}2^{qj(\gamma_{1}+\frac{n}{2})}|a^{\varepsilon}_{j,k}|^{q}\chi(2^{j}\cdot-k)\Big)^{1/q}
\Big\|_{L^{p}}\\
&\leq
C\inf\limits_{P_{\tilde{Q}, f}\in S^{\gamma_{1},\gamma_{2}}_{p,q,f}}
\|\varphi_{{\tilde{Q}}} (f-P_{\tilde{Q}, f})\|_{\dot{F}^{\gamma_{1},
q}_{p}}.
\end{array}$$
Hence (\ref{eq:c}) holds.

Conversely, for any cube $Q$, there exist $2^{n}$ dyadic cubes $Q_{i}$ such that
$2^{M+2}l(Q)\leq l(Q_{i})\leq 2^{M+3}l(Q)$
and
$$\begin{array}{rl}
&\phi_{Q}(x) f(x)= \phi_{Q}(x) \sum\limits^{2^n}_{i=1}
\sum\limits_{\epsilon\in E_{n},Q_{j,k}\subset Q_{i}}
f^{\epsilon}_{j,k} \Phi^{\epsilon}_{j,k}(x). \end{array}$$
 Hence
$$\begin{array}{rcl}
\|\phi_{Q} f\|_{\dot{F}^{\gamma_{1}, q}_{p}}& \leq & \Big\|\phi_{Q}
\sum\limits^{2^n}_{i=1} \sum\limits_{\epsilon\in
E_{n},Q_{j,k}\subset Q_{i}} f^{\epsilon}_{j,k} \Phi^{\epsilon}_{j,k}
\Big\|_{\dot{F}^{\gamma_{1}, q}_{p}}\\
&\leq & C\Big\|\sum\limits^{2^n}_{i=1} \sum\limits_{\epsilon\in
E_{n},Q_{j,k}\subset Q_{i}} f^{\epsilon}_{j,k}
\Phi^{\epsilon}_{j,k}\Big\|_{\dot{F}^{\gamma_{1}, q}_{p}}.
\end{array}$$
If  (\ref{eq:c}) holds, we get (\ref{eq:b}).
\end{proof}
\begin{remark}
If $\gamma_{2}=\frac{n}{p}$, $\dot{F}^{\gamma_{1},\gamma_{2}}_{p,q}$ becomes the Triebel-Lizorkin space $\dot{F}^{\gamma_{1},q}_{p}$.
Moreover, if $\gamma_{2}>\frac{n}{p}$, for any $f\in \dot{F}^{\gamma_{1},\gamma_{2}}_{p,q}$, the wavelet coefficients of $f$ are all zero. Hence $f$ is only a polynomial.
\end{remark}

By  Theorem \ref{lem:c}, we can identify a function with its wavelet
coefficients. That is to say,
$$\begin{array}{rl}
&\sum\limits_{(\epsilon,j,k)\in \Lambda_n}a^{\epsilon}_{j,k}
\Phi^{\epsilon}_{j,k}(x)\in
\dot{F}^{\gamma_{1},\gamma_{2}}_{p,q}\text{ if and only if }
\{a^{\epsilon}_{j,k}\}_{(\epsilon,j,k)\in \Lambda_n} \text{ satisfy
} (\ref{eq:c}). \end{array}$$
 Further, it is easy to
check the following results about Triebel type oscillation spaces.
\begin{corollary}\label{coro-Banach}
\label{cor:BMpropty} Given $0< p<\infty, 0<q\leq \infty$ and $\gamma_{1},
\gamma_{2}\in\mathbb{R}$.
\!\!\!\!
\begin{itemize}
\item[(i)]  The definition of $\dot{F}^{\gamma_{1},\gamma_{2}}_{p, q}$ is independent of the
choice of $\phi$.
\item[(ii)]
$\dot{F}^{\gamma_{1},\gamma_{2}}_{p,q}(\mathbb{R}^{n})$ are
Banach spaces for $1\leq p<\infty$ and $1\leq q\leq \infty$.

\end{itemize}
\end{corollary}

\subsection{Calder\'on-Zygmund operators}
Now we introduce some preliminaries about Calder\'on-Zygmund
operators. See \cite{Me, MY}. For $x\neq y$, let $K(x,y)$
be a smooth function such that there exists a sufficiently large
$N_{0}\leq m_0$ satisfying that
\begin{equation}\label{regular-condition}
|\partial ^{\alpha}_{x}\partial ^{\beta}_{y}
K(x,y)| \leq \frac{C}{|x-y|^{(n+|\alpha|+|\beta|)}}, \forall |\alpha|+
|\beta|\leq N_{0}.
\end{equation}
\begin{definition}
 A linear operator
 $$Tf(x)=\int K(x,y) f(y) dy$$
 is said to be a
Calderon-Zygmund operator if
\begin{itemize}
\item[(i)] $T$ is continuous from
$C^{1}(\mathbb{R}^{n})$ to $(C^{1}(\mathbb{R}^{n}))'$;
\item[(ii)] the kernel $K$ satisfies (\ref{regular-condition});
\item[(iii)] $Tx^{\alpha}=T^{*}x^{\alpha}=0, \forall \alpha \in \mathbb{N}^{n}$.
\end{itemize}
 We denote by $ CZO(N_{0})$ the set of all operators satisfying (i), (ii) and (iii).
\end{definition}

From (\ref{regular-condition}), we can see that the kernel $K(\cdot,\cdot)$ may have high singularity on the diagonal $x=y$.
According to Schwartz kernel theorem,
$K(\cdot,\cdot)$ is a distribution in $S'(\mathbb{R}^{2n})$. For any $
(\epsilon,j,k), (\epsilon',j',k')\in \Lambda_{n}$, let
$$a^{\epsilon,\epsilon'}_{j,k,j',k'}= \langle K(\cdot,\cdot),
\Phi^{\epsilon}_{j,k} \Phi^{\epsilon'}_{j',k'}\rangle.$$ If $T\in  CZO(N_{0})$, its kernel $K(\cdot,\cdot)$ and $\{a^{\epsilon,\epsilon'}_{j,k,j',k'}\}$ satisfy the following relations. We refer the reader to
Meyer \cite{Me}, Meyer-Yang \cite{MY} and Yang \cite{Yang1} for the proofs.

\begin{lemma} \label{lem:CZO}
(i) If $T\in CZO(N_{0})$, then  the coefficients
$\{a^{\epsilon,\epsilon'}_{j,k,j',k'}\}_{(\epsilon,j,k), (\epsilon',j',k')\in \Lambda_{n}}$ satisfy the following
condition:
\begin{equation}\label{eq:a}
|a^{\epsilon,\epsilon'}_{j,k,j',k'}| \leq C
2^{-|j-j'|(\frac{n}{2}+N_{0})}
\Big(\frac{2^{-j}+2^{-j'}}{2^{-j}+2^{-j'}
+|k2^{-j}-k'2^{-j'}|}\Big)^{n+N_{0}}.
\end{equation}

(ii) If $\{a^{\epsilon,\epsilon'}_{j,k,j',k'}\}_{(\epsilon,j,k), (\epsilon',j',k')\in \Lambda_{n}}$ satisfy  (\ref{eq:a}), then
$$\begin{array}{rl}
&K(x,y)=\sum\limits_{(\epsilon,j,k)\in\Lambda_{n}}\sum\limits_{
(\epsilon',j',k')\in \Lambda_{n}}
a^{\epsilon,\epsilon'}_{j,k,j',k'}\Phi^{\epsilon}_{j,k}(x)
\Phi^{\epsilon'}_{j',k'}(y)\end{array}$$
 in the sense of distributions
and for any small positive real number $\delta$, $T\in
CZO(N_{0}-\delta)$.

\end{lemma}
For $A>0$ and a sequence $f=:\{f_{j}\}$, we define the vector-valued maximal function $M_{A}(f)$ as
$$M_{A}(f)(x)=\Big(\sum_{j}M(|f_{j}|^{A})(x)\Big)^{1/A}.$$

 For $\{a^{\varepsilon}_{j,k}\}_{(\varepsilon,j,k)\in\Lambda_{n}}$, we set
$$f_{j'}=\sum\limits_{(\varepsilon,j,k)\in\Lambda_{n}}2^{j(s+\frac{n}{2})}|a^{\varepsilon}_{j,k}|\chi(2^{j}x-k).$$
Let
$$
g^{k}_{j,j'}=\left\{ \begin{aligned}
&\sum_{\varepsilon',k'}\frac{2^{j'(s+\frac{n}{2})}|a^{\varepsilon'}_{j',k'}|}{(1+|k'-2^{j'-j}k|)^{n+\gamma}},\ j\geq j', k\in\mathbb{Z}^{n};\\
&\sum_{\varepsilon',k'}\frac{2^{j'(s+\frac{n}{2})}|a^{\varepsilon'}_{j',k'}|}{(1+|k-2^{j-j'}k'|)^{n+\gamma}},\ j< j', k\in\mathbb{Z}^{n}.
\end{aligned}
\right.$$
Yang \cite{Yang1} obtained the following result.
\begin{lemma}{\rm (\cite{Yang1}, Chapter 5, Lemma 3.2)}
For any $\gamma>\frac{n}{A}+1$ and $x\in Q_{j,k}$, we have
$$g^{k}_{j,j'}=\left\{ \begin{aligned}
&CM_{A}(f_{j'})(x),\ \text{ if }j\geq j';\\
&C2^{\frac{n(j'-j)}{A}}M_{A}(f_{j'})(x),\ \text{ if } j< j'.
\end{aligned}
\right.$$
\end{lemma}
Following the idea of  \cite{MY, Yang1, YSY}, we
can prove that the Calder\'on-Zygmund operators are bounded on
$\dot{F}^{\gamma_{1},\gamma_{2}}_{p,q}$. For completeness, we give
the proof. In fact, for all $(\epsilon,j,k)\in
\Lambda_{n}$, denote
$$\begin{array}{rl}&\tilde{g}^{\epsilon}_{j,k} =\sum \limits_{
(\epsilon',j',k')\in \Lambda_{n}} a^{\epsilon,\epsilon'}_{j,k,j',k'}
g^{\epsilon'}_{j',k'}.\end{array}$$
By Lemma \ref{lem:CZO}, the boundedness of Calder\'on-Zygmund operators on $\dot{F}^{\gamma_{1}, \gamma_{2}}_{p,q}$ is equivalent to the following theorem.

\begin{theorem}\label{lem:CZcontinuity}
Given $0<p,q\leq \infty$ and $ \gamma_{1}, \gamma_{2}\in
\mathbb{R}$. There exists sufficiently big $N_0$ such that
 $\{a^{\epsilon,\epsilon'}_{j,k,j',k'} \}_{(\epsilon,j,k),\,
(\epsilon',j',k')\in \Lambda_{n}}$ satisfies
(\ref{eq:a}).  If $ \{g^{\epsilon}_{j,k}\}_{(\epsilon,j,k)\in
\Lambda_n} \in \dot{F}^{\gamma_{1},\gamma_{2}}_{p,q}$,
$\{\tilde{g}^{\epsilon}_{j,k}\}_{(\epsilon,j,k)\in \Lambda_n} \in
\dot{F}^{\gamma_{1},\gamma_{2}}_{p,q}$. Consequently,
Calder\'on-Zygmund operators are bounded on
$\dot{F}^{\gamma_{1},\gamma_{2}}_{p,q}$.
\end{theorem}

\begin{proof}
Let $Q$ be any dyadic cube with $|Q|=2^{-nj_{0}}$. For $\tau\geq1$, denote by $Q_{\tau}$ the dyadic cube satisfying $Q\subset Q_{\tau}$ and $|Q_{\tau}|=2^{n\tau}|Q|$. Specially, $Q_{0}=Q$. If $l\in\mathbb{Z}^{n}$ and $Q_{j',k'}\subset 2^{-j_{0}}l+Q_{\tau}$, we denote $Q_{j',k'}\in S_{\tau, l}$.
Then
we have
$$\begin{array}{rl}
\tilde{g}^{\varepsilon}_{j,k}&=\sum\limits_{(\varepsilon',j',k')\in\Lambda_{n}}a^{\varepsilon,\varepsilon'}_{j,k,j',k'}g^{\varepsilon'}_{j',k'}\\
&=\sum\limits_{\tau\geq0}\sum\limits_{l\in\mathbb{Z}^{n}}\sum\limits_{j'}\sum\limits_{Q_{j',k'}\in S_{\tau,l}}a^{\varepsilon,\varepsilon'}_{j,k,j',k'}g^{\varepsilon'}_{j',k'}.
\end{array}$$
We will prove
$$\begin{array}{rl}
I_{Q_{r}}=&|Q_{r}|^{\frac{\gamma_{2}}{n}-\frac{1}{p}}\Big\|\Big[\sum\limits_{j\geq-\log_{2}r}\sum\limits_{Q_{j,k}\subset Q_{r}}2^{qj(\gamma_{1}+\frac{n}{2})}\Big|\tilde{g}^{\varepsilon}_{j,k}\Big|^{q}\chi(2^{j}x-k)\Big]^{1/q}\Big\|_{p}<\infty.
\end{array}$$
By H\"older's inequality, for $\delta$ small enough, we have
$$\begin{array}{rl}
I_{Q_{r}}=&|Q_{r}|^{\frac{\gamma_{2}}{n}-\frac{1}{p}}\Big\|\Big[\sum\limits_{j\geq-\log_{2}r}\sum\limits_{Q_{j,k}\subset Q_{r}}2^{qj(\gamma_{1}+\frac{n}{2})}\Big|\sum\limits_{\tau\geq0}\sum\limits_{l\in\mathbb{Z}^{n}}\\
&\sum\limits_{j'}\sum\limits_{Q_{j',k'}\in S_{\tau,l}}a^{\varepsilon,\varepsilon'}_{j,k,j',k'}g^{\varepsilon'}_{j',k'}\Big|^{q}\chi(2^{j}x-k)\Big]^{1/q}\Big\|_{p}\\
\lesssim&|Q_{r}|^{\frac{\gamma_{2}}{n}-\frac{1}{p}}\Big\|\Big[\sum\limits_{j\geq-\log_{2}r}\sum\limits_{Q_{j,k}\subset Q_{r}}2^{qj(\gamma_{1}+\frac{n}{2})}\sum\limits_{\tau\geq0}\sum\limits_{l\in\mathbb{Z}^{n}}2^{\tau\delta}(1+|l|)^{n+\delta}\\
&\Big|\sum\limits_{j'}\sum\limits_{Q_{j',k'}\in S_{\tau,l}}a^{\varepsilon,\varepsilon'}_{j,k,j',k'}g^{\varepsilon'}_{j',k'}\Big|^{q}\chi(2^{j}x-k)\Big]^{1/q}\Big\|_{p}.
\end{array}$$

(1) For $\tau=0$,
$$\begin{array}{rl}
I_{Q_{r}}\lesssim&|Q_{r}|^{\frac{\gamma_{2}}{n}-\frac{1}{p}}\Big\|\Big[\sum\limits_{j\geq-\log_{2}r}\sum\limits_{Q_{j,k}\subset Q_{r}}2^{qj(\gamma_{1}+\frac{n}{2})}\sum\limits_{l\in\mathbb{Z}^{n}}(1+|l|)^{n+\delta}\\
&\Big|\sum\limits_{j'}\sum\limits_{Q_{j',k'}\in S_{0,l}}a^{\varepsilon,\varepsilon'}_{j,k,j',k'}g^{\varepsilon'}_{j',k'}\Big|^{q}\chi(2^{j}x-k)\Big]^{1/q}\Big\|_{p}.
\end{array}$$
If $|l|<8^{n}$, by the boundedness of Calder\'on-Zygmund operator on $\dot{F}^{\gamma_{1}, q}_{p}$, we have
$$\begin{array}{rl}
&|Q_{r}|^{\frac{\gamma_{2}}{n}-\frac{1}{p}}\Big\|\Big[\sum\limits_{j\geq-\log_{2}r}\sum\limits_{Q_{j,k}\subset Q_{r}}2^{qj(\gamma_{1}+\frac{n}{2})}\sum\limits_{|l|<8^{n}}(1+|l|)^{n+\delta}\\
&\quad\Big|\sum\limits_{j'}\sum\limits_{Q_{j',k'}\in S_{0,l}}a^{\varepsilon,\varepsilon'}_{j,k,j',k'}g^{\varepsilon'}_{j',k'}\Big|^{q}\chi(2^{j}x-k)\Big]^{1/q}\Big\|_{p}\\
&\lesssim|Q_{r}|^{\frac{\gamma_{2}}{n}-\frac{1}{p}}
\Big\|\sum\limits_{j'}\sum\limits_{Q_{j',k'}\in S_{0,l}}g^{\varepsilon'}_{j',k'}\Phi^{\varepsilon'}_{j',k'}\Big\|_{\dot{F}^{\gamma_{1}, q}_{p}}<\infty.
\end{array}$$

If $|l|>8^{n}$, because
$$|a^{\varepsilon,\varepsilon'}_{j,k,j',k'}|\lesssim 2^{-|j-j'|(\frac{n}{2}+N_{0})}\Big(\frac{2^{-j}+2^{-j'}}{2^{-j}+2^{-j'}+|2^{-j}k-2^{-j'}k'|}\Big)^{n+N_{0}},$$
$$\begin{array}{rl}
I_{Q_{r}}\lesssim&|Q_{r}|^{\frac{\gamma_{2}}{n}-\frac{1}{p}}\Big\|\Big[\sum\limits_{j\geq-\log_{2}r}\sum\limits_{Q_{j,k}\subset Q_{r}}2^{qj(\gamma_{1}+\frac{n}{2})}\sum\limits_{|l|>8^{n}}(1+|l|)^{n+\delta}\\
&\Big|\sum\limits_{j\geq j'}\sum\limits_{Q_{j',k'}\in S_{0,l}}a^{\varepsilon,\varepsilon'}_{j,k,j',k'}g^{\varepsilon'}_{j',k'}\Big|^{q}\chi(2^{j}x-k)\Big]^{1/q}\Big\|_{p}\\
+&|Q_{r}|^{\frac{\gamma_{2}}{n}-\frac{1}{p}}\Big\|\Big[\sum\limits_{j\geq-\log_{2}r}\sum\limits_{Q_{j,k}\subset Q_{r}}2^{qj(\gamma_{1}+\frac{n}{2})}\sum\limits_{|l|>8^{n}}(1+|l|)^{n+\delta}\\
&\Big|\sum\limits_{j<j'}\sum\limits_{Q_{j',k'}\in S_{0,l}}a^{\varepsilon,\varepsilon'}_{j,k,j',k'}g^{\varepsilon'}_{j',k'}\Big|^{q}\chi(2^{j}x-k)\Big]^{1/q}\Big\|_{p}\\
:=&I_{1}+I_{2}.
\end{array}$$
For $I_{1}$,
$$\begin{array}{rl}
|a^{\varepsilon,\varepsilon'}_{j,k,j',k'}|\lesssim&2^{-|j-j'|(\frac{n}{2}+N_{0})}(1+|2^{j'-j}k-k'|)^{-(n+N_{0})}\\
\lesssim&2^{-|j-j'|(\frac{n}{2}+N_{0})}(2^{j'-j_{0}}|l|)^{-(n+N_{0}/2)}(1+|2^{j'-j}k-k'|)^{-N_{0}/2}.
\end{array}$$
On the other hand, $Q_{j',k'}\in S_{0,l}$ implies $Q_{j',k'}=2^{-j}l+Q_{j,k}$. It is easy to see that $|k2^{j'-j}-k'|\sim 2^{j'-j_{0}}|l|$. Then
$$\begin{array}{rl}
I_{1}\lesssim&|Q_{r}|^{\frac{\gamma_{2}}{n}-\frac{1}{p}}\Big\|\Big[\sum\limits_{j\geq-\log_{2}r}\sum\limits_{Q_{j,k}\subset Q_{r}}2^{qj(\gamma_{1}+\frac{n}{2})}\sum\limits_{|l|>8^{n}}(1+|l|)^{n+\delta}\\
&\Big|\sum\limits_{j\geq j'}\sum\limits_{Q_{j',k'}\in S_{0,l}}
2^{-|j-j'|(\frac{n}{2}+N_{0})}(2^{j'-j_{0}}|l|)^{-(n+N_{0}/2)}\\
&(1+|2^{j'-j}k-k'|)^{-N_{0}/2}
|g^{\varepsilon'}_{j',k'}|\Big|^{q}\chi(2^{j}x-k)\Big]^{1/q}\Big\|_{p}\\
\lesssim&|Q_{r}|^{\frac{\gamma_{2}}{n}-\frac{1}{p}}\Big\|\Big[\sum\limits_{j\geq-\log_{2}r}\sum\limits_{Q_{j,k}\subset Q_{r}}2^{qj(\gamma_{1}+\frac{n}{2})}\sum\limits_{j\geq j'}2^{\delta(j-j')}2^{-q(j-j')(\frac{n}{2}+N_{0})}\\
&2^{-q(n+N_{0}/2)(j'-j_{0})}\Big(\sum\limits_{Q_{j',k'}\subset Q^{l}_{r}}\frac{|g^{\varepsilon'}_{j',k'}|}{(1+|2^{j'-j}k-k'|)^{N_{0}/2}}
\Big)^{q}\chi(2^{j}x-k)\Big]^{1/q}\Big\|_{p}\\
\lesssim&|Q_{r}|^{\frac{\gamma_{2}}{n}-\frac{1}{p}}\sum\limits_{|l|>8^{n}}(1+|l|)^{-N'}\Big\|\Big[\sum\limits_{j\geq-\log_{2}r}\sum\limits_{Q_{j,k}\subset Q_{r}}\sum\limits_{j>j'}2^{(j-j')[\delta+q(\gamma_{1}+\frac{n}{2})-q(\frac{n}{2}+N_{0})]}\\
&(M_{A}(f_{j'})(x))^{q}\chi(2^{j}x-k)\Big]^{1/q}\Big\|_{p},
\end{array}$$
where
$$\begin{array}{rl}
f_{j'}=2^{j'(\gamma_{1}+\frac{n}{2})}\sum\limits_{Q_{j',k'}\subset Q^{l}_{r}}|g^{\varepsilon'}_{j',k'}|\chi(2^{j'}x-k').
\end{array}$$
We can get
$$\begin{array}{rl}
&I_{1}\\
&\lesssim |Q_{r}|^{\frac{\gamma_{2}}{n}-\frac{1}{p}}\sum\limits_{|l|>8^{n}}(1+|l|)^{-N'}\Big\|\Big[\sum\limits_{j\geq-\log_{2}r}\sum\limits_{Q_{j,k}\subset Q_{r}}\sum\limits_{j>j'}2^{(j-j')[\delta+q(\gamma_{1}+\frac{n}{2})-q(\frac{n}{2}+N_{0})]}\\
&\quad \Big|2^{j'(\gamma_{1}+\frac{n}{2})}\sum\limits_{Q_{j',k'}\subset Q^{l}_{r}}|g^{\varepsilon'}_{j',k'}|\chi(2^{j'}x-k')\Big|^{q}\chi(2^{j}x-k)\Big]^{1/q}\Big\|_{p}\\
&\lesssim |Q_{r}|^{\frac{\gamma_{2}}{n}-\frac{1}{p}}\sum\limits_{|l|>8^{n}}(1+|l|)^{-N'}\Big\|\Big[\sum\limits_{j'\geq-\log_{2}r}
2^{qj'(\gamma_{1}+\frac{n}{2})}\sum\limits_{Q_{j',k'}\in S_{0, l}}|g^{\varepsilon'}_{j',k'}|^{q}\chi(2^{j'}x-k')\Big]^{1/q}\Big\|_{p}\\
&\lesssim \|g\|_{\dot{F}^{\gamma_{1},\gamma_{2}}_{p,q,}}.
\end{array}$$

For $I_{2}$, because $j<j'$, we have
$$\begin{array}{rl}
|a^{\varepsilon,\varepsilon'}_{j,k,j',k'}|\lesssim&2^{-|j-j'|(\frac{n}{2}+N_{0})}\Big(\frac{2^{-j}+2^{-j'}}{2^{-j}+2^{-j'}+|2^{-j}k-2^{-j'}k'|}\Big)^{n+N_{0}}\\
\lesssim&2^{-(j'-j)(\frac{n}{2}+N_{0})}\Big(1+|k-2^{j-j'}k'|\Big)^{-(n+N_{0})}.
\end{array}$$
Let $x_{0}$ be the center of $Q$.
$$\begin{array}{rl}
|k-2^{j-j'}k'|&=2^{j}|2^{-j}k-x_{0}-2^{-j_{0}}l+x_{0}+2^{-j_{0}}l-2^{-j'}k'|\\
&\geq 2^{j}(2^{-j_{0}}|l|-2^{-j_{0}}-2^{-j_{0}})\\
&\geq 2^{j-j_{0}}|l|.
\end{array}$$
Let
$$\begin{array}{rl}
f_{j'}=2^{j'(\gamma_{1}+\frac{n}{2})}\sum\limits_{Q_{j',k'}\in S_{0, l}}|g^{\varepsilon'}_{j',k'}|\chi(2^{j'}x-k').
\end{array}$$
Hence we get
$$\begin{array}{rl}
I_{2}\lesssim&\sum\limits_{|l|>8^{n}}(1+|l|)^{\frac{n+\delta}{q}-(n+\frac{N_{0}}{2})}|Q_{r}|^{\frac{\gamma_{2}}{n}-\frac{1}{p}}\Big\|\Big[\sum\limits_{j\geq j_{0}}\sum\limits_{Q_{j,k}\subset Q_{r}}2^{qj(\gamma_{1}+\frac{n}{2})}\\
&\Big|\sum\limits_{j'>j}2^{-(j'-j)(\frac{n}{2}+N_{0})}2^{-(n+\frac{N_{0}}{2})(j-j_{0})}\sum\limits_{Q_{j',k'}\in S_{0,l}}\frac{|g^{\varepsilon'}_{j',k'}|}{(1+|2^{j-j'}k'-k|)^{N_{0}/2}}\Big|^{q}\chi(2^{j}x-k)\Big]^{1/q}\Big\|_{p}\\
\lesssim&\sum\limits_{|l|>8^{n}}(1+|l|)^{-N'}|Q_{r}|^{\frac{\gamma_{2}}{n}-\frac{1}{p}}\Big\|\Big[\sum\limits_{j\geq j_{0}}\sum\limits_{Q_{j,k}\subset Q_{r}}2^{qj(\gamma_{1}+\frac{n}{2})}\sum\limits_{j'>j}2^{\delta(j'-j)}2^{-q(j'-j)(\frac{n}{2}+N_{0})}\\
&2^{-q(n+\frac{N_{0}}{2})(j-j_{0})}\Big|\sum\limits_{Q_{j',k'}\in S_{0,l}}\frac{|g^{\varepsilon'}_{j',k'}|}{(1+|2^{j-j'}k'-k|)^{N_{0}/2}}\Big|^{q}\chi(2^{j}x-k)\Big]^{1/q}\Big\|_{p}\\
\lesssim&\sum\limits_{|l|>8^{n}}(1+|l|)^{-N'}|Q_{r}|^{\frac{\gamma_{2}}{n}-\frac{1}{p}}\Big\|\Big[\sum\limits_{j\geq j_{0}}\sum\limits_{Q_{j,k}\subset Q_{r}}2^{qj(\gamma_{1}+\frac{n}{2})}\sum\limits_{j'>j}2^{\delta(j'-j)}2^{-q(j'-j)(\frac{n}{2}+N_{0})}\\
&2^{-q(n+\frac{N_{0}}{2})(j-j_{0})}2^{-qj'(\gamma_{1}+\frac{n}{2})}2^{qn(j'-j)/A}\Big|M_{A}(f_{j'})(x)\Big|^{q}\chi(2^{j}x-k)\Big]^{1/q}\Big\|_{p}\\
\lesssim& |Q_{r}|^{\frac{\gamma_{2}}{n}-\frac{1}{p}}\sum\limits_{|l|>8^{n}}(1+|l|)^{-N'}\Big\|\Big[\sum\limits_{j'\geq-\log_{2}r}
2^{qj'(\gamma_{1}+\frac{n}{2})}\sum\limits_{Q_{j',k'}\in S_{0, l}}|g^{\varepsilon'}_{j',k'}|^{q}\chi(2^{j'}x-k')\Big]^{1/q}\Big\|_{p}\\
\lesssim& \|g\|_{\dot{F}^{\gamma_{1},\gamma_{2}}_{p,q,}}.
\end{array}$$

(2) For $\tau\geq1$, we can get for $Q_{j,k}\subset Q$ and $Q_{j',k'}=2^{\tau-j_{0}}l+2^{\tau}Q$, $j'=j_{0-\tau}$. it is easy to see that
$j>j'$ for this case. If $|l|<8^{n}$ and $j>j'$,
$$\begin{array}{rl}
|a^{\varepsilon,\varepsilon'}_{j,k,j',k'}|\lesssim&2^{-|j-j'|(\frac{n}{2}+N_{0})}\Big(\frac{2^{-j}+2^{-j'}}{2^{-j}+2^{-j'}+|2^{-j}k-2^{-j'}k'|}\Big)^{n+N_{0}}\\
\lesssim&2^{-|j-j'|(\frac{n}{2}+N_{0})}\Big(\frac{2^{j'-j}+1}{2^{j'-j}+1+|2^{j'-j}k-k'|}\Big)^{n+N_{0}}\\
\lesssim&2^{-|j-j'|(\frac{n}{2}+N_{0})}2^{-\frac{\tau N_{0}}{2}}(1+|k'-2^{j'-j}k|)^{-(n+\frac{N_{0}}{2})}.
\end{array}$$
Let
$$\begin{array}{rl}
f_{j'}=2^{j'(\gamma_{1}+\frac{n}{2})}\sum\limits_{Q_{j',k'}\in S_{\tau, l}}|g^{\varepsilon'}_{j',k'}|\chi(2^{j'}x-k').
\end{array}$$
Hence
$$\begin{array}{rl}
I_{Q_{r}}\lesssim&\sum\limits_{|l|<8^{n}}(1+|l|)^{n+\delta}\sum\limits_{\tau\geq1}2^{\tau\delta}|Q_{r}|^{\frac{\gamma_{2}}{n}-\frac{1}{p}}
\Big\|\Big[\sum\limits_{j\geq j_{0}}\sum\limits_{Q_{j,k}\subset Q_{r}}2^{qj(\gamma_{1}+\frac{n}{2})}\\
&\Big|\sum\limits_{j>j'}\sum\limits_{Q_{j',k'}\in S_{\tau,l}}2^{-|j-j'|(\frac{n}{2}+\frac{N_{0}}{2})}2^{-\frac{\tau N_{0}}{2}}
\frac{|g^{\varepsilon'}_{j',k'}|}{(1+|k'-2^{j'-j}k|)^{(n+\frac{N_{0}}{2})}}\Big|^{q}\chi(2^{j}x-k)\Big]^{1/q}\Big\|_{p}\\
\lesssim&\sum\limits_{\tau\geq1}2^{\tau(\delta-\frac{N_{0}}{2})}\sum\limits_{|l|<8^{n}}(1+|l|)^{n+\delta}|Q_{r}|^{\frac{\gamma_{2}}{n}-\frac{1}{p}}
\Big\|\Big[\sum\limits_{j\geq j_{0}}\sum\limits_{Q_{j,k}\subset Q_{r}}2^{qj(\gamma_{1}+\frac{n}{2})}\\
&\sum\limits_{j>j'}2^{\delta(j-j')}2^{-q(j-j')(\frac{n}{2}+N_{0})}\Big|\sum\limits_{Q_{j',k'}\in S_{\tau, l}}\frac{|g^{\varepsilon'}_{j',k'}|}{(1+|k'-2^{j'-j}k|)^{(n+\frac{N_{0}}{2})}}\Big|^{q}\chi(2^{j}x-k)\Big]^{1/q}\Big\|_{p}\\
\lesssim&\sum\limits_{\tau\geq1}2^{\tau(\delta-\frac{N_{0}}{2})}\sum\limits_{|l|<8^{n}}(1+|l|)^{n+\delta}|Q_{r}|^{\frac{\gamma_{2}}{n}-\frac{1}{p}}
\Big\|\Big[\sum\limits_{j\geq j_{0}}\sum\limits_{Q_{j,k}\subset Q_{r}}\sum\limits_{j>j'}2^{(j-j')[\delta-qN_{0}+q\gamma_{1}]}\\
&(M_{A}(f_{j'})(x))^{q}\chi(2^{j}x-k)\Big]^{1/q}\Big\|_{p}\\
\lesssim&\|g\|_{\dot{F}^{\gamma_{1},\gamma_{2}}_{p,q,}}.
\end{array}$$

If $|l|>8^{n}$, for $Q_{j',k'}=2^{\tau-j_{0}}l+Q_{\tau}$,
$$|2^{j'-j}k-k'|=2^{j'}|2^{-j}k-2^{-j'}k'|\geq 2^{\tau+j'-j_{0}}|l|.$$
Then
$$|a^{\varepsilon,\varepsilon'}_{j,k,j',k'}|\lesssim2^{-|j-j'|(\frac{n}{2}+N_{0})}(2^{\tau+j'-j_{0}}|l|)^{-n+N_{0}/2}(1+|2^{j'-j}k-k'|)^{-\frac{N_{0}}{2}}.$$
 Let
$$\begin{array}{rl}
f_{j'}=2^{j'(\gamma_{1}+\frac{n}{2})}\sum\limits_{Q_{j',k'}\in S_{\tau, l}}|g^{\varepsilon'}_{j',k'}|\chi(2^{j'}x-k').
\end{array}$$
So
$$\begin{array}{rl}
I_{Q_{r}}\lesssim&\sum\limits_{\tau\geq1}2^{\tau(\delta-\frac{N_{0}}{2})}\sum\limits_{|l|>8^{n}}(1+|l|)^{-N'}|Q_{r}|^{\frac{\gamma_{2}}{n}-\frac{1}{p}}
\Big\|\Big[\sum\limits_{j\geq j_{0}}\sum\limits_{Q_{j,k}\subset Q_{r}}2^{qj(\gamma_{1}+\frac{n}{2})}\\
&\sum\limits_{j>j'}2^{\delta(j-j')}2^{-q(j-j')(\frac{n}{2}+N_{0})}2^{-q(j'-j_{0})N_{0}/2}\Big(\sum\limits_{Q_{j',k'}\in S_{\tau, l}}|g^{\varepsilon'}_{j',k'}|\\
&(1+|2^{j'-j}k-j'|)^{-(n+\frac{N_{0}}{2})}\Big)^{q}\chi(2^{j}x-k)\Big]^{1/q}\Big\|_{p}\\
\lesssim&\sum\limits_{\tau\geq1}2^{\tau(\delta-\frac{N_{0}}{2})}\sum\limits_{|l|>8^{n}}(1+|l|)^{-N'}|Q_{r}|^{\frac{\gamma_{2}}{n}-\frac{1}{p}}
\Big\|\Big[\sum\limits_{j\geq j_{0}}\sum\limits_{Q_{j,k}\subset Q_{r}}\sum\limits_{j>j'}2^{(j-j')[\delta+\frac{qn}{2}+q\gamma_{1}-\frac{qn}{2}-\frac{qN_{0}}{2}]}\\
&\Big(M_{A}(f_{j'})(x)\Big)^{q}\chi(2^{j}x-k)\Big]^{1/q}\Big\|_{p}\\
\lesssim&\|g\|_{\dot{F}^{\gamma_{1},\gamma_{2}}_{p,q,}}.
\end{array}$$

\end{proof}

For $i=1,2$ and two regular orthogonal wavelets basis
$\{\Phi^{i,\epsilon}_{j,k}(x)\}_{(\epsilon,j,k)\in \lambda_n}$,
$\forall (\epsilon,j,k), (\epsilon',j',k')\in \Lambda_{n}$, denote
$a^{\epsilon,\epsilon'}_{j,k,j',k'}= \langle
\Phi^{1,\epsilon}_{j,k}, \Phi^{2,\epsilon'}_{j',k'}\rangle$.
We know that $\{a^{\epsilon,\epsilon'}_{j,k,j',k'}
\}_{(\epsilon,j,k),\, (\epsilon',j',k')\in \Lambda_{n}}$ satisfies
the condition (\ref{eq:a}). According to Theorem
\ref{lem:CZcontinuity},  Lemma \ref{lem:c} is also true for Meyer
wavelets. The reader can also find a proof of the following result
in \cite{YSY}.
\begin{lemma}\label{lem:Me}
The wavelet characterization in Lemma \ref{lem:c} is also true for Meyer wavelets.
\end{lemma}
\section{ Triebel  type tent spaces}\label{sec3}
For the rest of this paper, we only use  classical tensorial Meyer wavelets.
In this section, we introduce two classes of Triebel type tent spaces which will be used in the well-posedness of Navier-Stokes equations.
 We would like
to remind the readers that, for wavelets $\{\Phi^{\epsilon}_{j,k}:
(\epsilon,j,k)\in \Lambda_{n}\}$, $2^{j}$ represents the range of
frequency $\xi$ and $2^{-j}k$ represents the range of position $x$
in some sense.

\subsection{Preliminaries relative to semigroup}

Throughout this section, we denote $N>0$ a fixed sufficient big real number.
For fixed $\beta>0$, we may choose a radial $\phi\in \mathscr S(\mathbb{R}^n)$
such that there exists $C_{\beta}>0$ satisfying
\begin{itemize}
\item[\rm(i)] $ \int_{\mathbb{R}^{n}} x^{\gamma} \phi(x) dx =0$\quad for all\ \ $\gamma\in \mathbb N^{n};$
\item[\rm(ii)] $ \int^{\infty}_{0} (\hat{\phi}(t^{\frac{1}{2\beta}}\xi))^{2} \frac{dt}{t} =1$\quad for all\ \ $\xi\neq 0;$
\item[\rm(iii)] $ \int^{\infty}_{0} \hat{\phi}( t^{\frac{1}{2\beta}}) e^{-t} \frac{dt}{t}=\frac 1{C_{\beta}}$.
\end{itemize}
See \cite[Lemma 1.1]{FJW} and \cite[Chap. 3, \S2]{Me}.

Define $\phi^{\beta}_{t}(x) = t^{-\frac{n}{2\beta}}\phi(t^{-\frac{1}{2\beta}} x)$.
Then $\hat{\phi}^{\beta}_{t}(\xi)= \hat{\phi}(t^{\frac{1}{2\beta}}\xi)$, and hence
\begin{equation*}
f(t,x):= e^{-t(-\Delta)^{\beta}} f(x) = K_{t}^{\beta}*f(x).
\end{equation*}
Since
$$\begin{array}{rl}
\hat{f}(\xi) &= C_{\beta}\int^{\infty}_{0} \hat{\phi}( t^{\frac{1}{2\beta}}) e^{-t} \frac{dt}{t} \hat{f}(\xi)\\
&=C_{\beta}\int^{\infty}_{0}
\hat{\phi}( t^{\frac{1}{2\beta}}|\xi|) e^{-t|\xi|^{2\beta}} \hat{f}(\xi) \frac{ dt}{t},\end{array}$$ we have
\begin{eqnarray}\label{Pull-back}
 \begin{array}{rl}
 f(x)& =C_{\beta}\int^{\infty}_{0}\int_{\mathbb{R}^n} f(t,x-y) \phi^{\beta}_{t}(y)dy\frac{dt}{t}:= \pi_{\phi}f(\cdot,x).\end{array}
\end{eqnarray}
For $(\epsilon, j,k) \in \Lambda_{n}$, let $a^{\epsilon}_{j,k}(t) =
\langle f(t,\cdot), \Phi^{\epsilon}_{j,k}\rangle$
and $a^{\epsilon}_{j,k} = \langle f, \Phi^{\epsilon}_{j,k}\rangle$. Then

$$\begin{array}{rl}
&f = \sum\limits_{(\epsilon, j,k)\in\Lambda_n} a^{\epsilon}_{j,k}
\Phi^{\epsilon}_{j,k} \text{ and } f(t,\cdot)= \sum\limits_{(\epsilon,
j,k)\in\Lambda_n} a^{\epsilon}_{j,k}(t) \Phi^{\epsilon}_{j,k} .
\end{array}$$

We first  express $a^{\epsilon}_{j,k}(t)$ by using
$a^{\epsilon'}_{j',k'}$. If $f(t,x)=K^{\beta}_{t}*f(x)$, then
$$\begin{array}{rl}
a^{\epsilon}_{j,k}(t)=& \sum\limits_{\epsilon',|j-j'|\leq 3, k'}
a^{\epsilon'}_{j',k'} \langle K^{\beta}_{t}
\Phi^{\epsilon'}_{j',k'},
\Phi^{\epsilon}_{j,k}\rangle\\
= &\sum\limits_{\epsilon',|j-j'|\leq 3, k'} a^{\epsilon'}_{j',k'}
\int e^{-t  |\xi|^{2\beta}} \hat{\Phi}^{\epsilon'} (2^{-j'}\xi)
\hat{\Phi}^{\epsilon} (2^{-j}\xi) e^{-i(2^{-j'}k'-2^{-j}k)\xi} d\xi\\
= &\sum\limits_{\epsilon',|j-j'|\leq 3, k'} a^{\epsilon'}_{j',k'}
\int e^{-t 2^{2j\beta} |\xi|^{2\beta}} \hat{\Phi}^{\epsilon'}
(2^{j-j'}\xi) \hat{\Phi}^{\epsilon} (\xi) e^{-i(2^{j-j'}k'-k)\xi}
d\xi.
\end{array}$$
Applying integration by parts,
we could control $a^{\epsilon}_{j,k}(t)$ by $\{ a^{\epsilon'}_{j',k'}\}$ as follows.

\begin{lemma}\label{lem3-1}
There exists a fixed small constant $\tilde c>0$ depending only on $\beta$ and
 the support of $\widehat{\Phi^{\epsilon}_{j,k}}$ such that
\begin{itemize}
\item[\rm(i)]  For $t 2^{2\beta j} \geq 1,$
$$\begin{array}{rl}
&|a^{\epsilon}_{j,k}(t)|\leq Ce^{-\tilde c t 2^{2j\beta}}
\sum\limits_{\epsilon',|j-j'|\leq 3, k'} |a^{\epsilon'}_{j',k'}|
(1+|2^{j-j'}k'-k)|)^{-N};\end{array}$$

\item[\rm(ii)] For $0\leq t 2^{2\beta j} \leq 1$,
$$\begin{array}{rl}
&|a^{\epsilon}_{j,k}(t)|\leq C \sum\limits_{|j-j'|\leq 3}
\sum\limits_{\epsilon',k'} |a^{\epsilon'}_{j',k'}|
(1+|2^{j-j'}k'-k|)^{-N}.\end{array}$$
\end{itemize}
\end{lemma}

Furthermore, if $f$ is obtained by (\ref{Pull-back}), then we could
express $a^{\epsilon}_{j,k}$ by $\{a^{\epsilon'}_{j',k'}(t)\}$ as
follows:

$$\begin{array}{rl}
& a^{\epsilon}_{j,k} = \int_{\mathbb{R}^{n+1}_{+}}
\sum\limits_{(\epsilon',j',k')\in\Lambda_n} a^{\epsilon'}_{j',k'}(t)
(\phi^{\beta}_{t}* \Phi^{\epsilon'}_{j',k'}(x))
\Phi^{\epsilon}_{j,k}(x)dx \frac{dt}{t}.
\end{array}$$
Similarly, we apply integration by party to obtain the following estimation.

\begin{lemma}
$$\begin{array}{rl}
& |a^{\epsilon}_{j,k}|\le C \sum\limits_{|j-j'|\le 3}
\int^{\infty}_{0} \big(\max\{t2^{2j' \beta}, t^{-1}2^{-2j'
\beta}\}\big)^{-N} \sum\limits_{(\epsilon',k')\in F_{n}}
\frac{|a^{\epsilon'}_{j',k'}(t) |} {(1+|2^{j-j'}k'-k|)^{N}
}\frac{dt}{t}.
\end{array}$$
\end{lemma}

\subsection{Tent
spaces $\mathbb{F}^{\gamma_{1}, \gamma_{2}}_{p,q,m,m'}$ and $\mathbb{F}^{\gamma}_{\tau,\infty}$} %
\label{sec:tent}
For any $a(t,x)$
defined on $\mathbb{R}^{n+1}_{+}$, by wavelet theory there exists a
family $\{a^{\epsilon}_{j,k}(t)\}_{(\epsilon,j,k)\in \Lambda_{n}}$
such that $$\begin{array}{rl}a(t,x)= \sum\limits_{(\epsilon,j,k)\in \Lambda_{n}}
a^{\epsilon}_{j,k}(t) \Phi^{\epsilon}_{j,k}(x).\end{array}$$ Given
$\gamma_{1},\gamma_{2}\in \mathbb{R}$, $1<p<\infty$, $m\in\mathbb{R}$ and $m'>0$.
$t$-Triebel-Lizorkin-Morrey spaces are defined as follows:
$$\begin{array}{rl}
\mathbb{F}^{\gamma_{1}, \gamma_{2}}_{p,q,m,m'}&=\mathbb{F}^{\gamma_{1}, \gamma_{2}, I}_{p,q,m }\cap \mathbb{F}^{\gamma_{1}, \gamma_{2},II}_{p,q}
\cap\mathbb{F}^{\gamma_{1}, \gamma_{2}, III}_{p,q,m}\cap \mathbb{F}^{\gamma_{1}, \gamma_{2}, IV}_{p,q, m'},
\end{array},$$
where the spaces of $L^{\infty}$ type norm on $t$ are defined as follows:
\begin{itemize}
\item[(i)] $f\in \mathbb{F}^{\gamma_{1}, \gamma_{2}, I}_{p,q, m}(\mathbb{R}^{n+1}_{+})$, if
$$\begin{array}{rl}
|Q_{r}|^{\frac{\gamma_{2}}{n}-\frac{1}{p}}t^{m}\Big\|\Big(\sum\limits_{j\geq\max\left\{-\log_{2}r, -\frac{\log_{2}t}{2\beta}\right\}}\sum\limits_{Q_{j,k}\subset Q_{r}}
2^{qj(\gamma_{1}+\frac{n}{2}+2m\beta)}|a^{\varepsilon}_{j,k}(t)|^{q}\chi(2^{j}x-k)\Big)^{1/q}\Big\|_{p}<\infty.
\end{array}$$
\item[(ii)] $f\in \mathbb{F}^{\gamma_{1}, \gamma_{2}, II}_{p,q}(\mathbb{R}^{n+1}_{+})$, if
$$\begin{array}{rl}
|Q_{r}|^{\frac{\gamma_{2}}{n}-\frac{1}{p}}\Big\|\Big(\sum\limits_{-\log_{2}r<j<-\frac{\log_{2}t}{2\beta}}\sum_{Q_{j,k}\subset Q_{r}}
2^{qj(\gamma_{1}+\frac{n}{2})}|a^{\varepsilon}_{j,k}(t)|^{q}\chi(2^{j}x-k)\Big)^{1/q}\Big\|_{p}<\infty.\end{array}$$
\end{itemize}
the spaces of integration type norm on $t$ are defined as follows:
\begin{itemize}
\item[(i)] $f\in \mathbb{F}^{\gamma_{1}, \gamma_{2}, III}_{p,q, m,m'}(\mathbb{R}^{n+1}_{+})$, if
$$\begin{array}{rl}
|Q_{r}|^{\frac{\gamma_{2}}{n}-\frac{1}{p}}\Big\|\Big(\sum\limits_{(\varepsilon,j,k)\in\Lambda_{Q}^{n}}
2^{qj(\gamma_{1}+\frac{n}{2}+2m\beta)}\int^{r^{2\beta}}_{2^{-2j\beta}}|a^{\varepsilon}_{j,k}(t)|^{q}t^{qm}\frac{dt}{t}\chi(2^{j}x-k)\Big)\Big\|_{p}<\infty.
\end{array}$$
\item[(ii)] $f\in \mathbb{F}^{\gamma_{1}, \gamma_{2}, IV}_{p,q,m'}(\mathbb{R}^{n+1}_{+})$, if
$$\begin{array}{rl}
|Q_{r}|^{\frac{\gamma_{2}}{n}-\frac{1}{p}}\Big\|\Big(\sum\limits_{(\varepsilon,j,k)\in\Lambda_{Q}^{n}}
2^{qj(\gamma_{1}+\frac{n}{2}+2m'\beta)}\int_{0}^{2^{-2j\beta}}|a^{\varepsilon}_{j,k}(t)|^{q}t^{qm'}\frac{dt}{t}\chi(2^{j}x-k)\Big)\Big\|_{p}<\infty.
\end{array}$$
\end{itemize}

Then, we define $t$-Bloch spaces and $t$-$L^{\infty}$ spaces which can be applied to the proof of boundedness of the bilinear operators.
For $\gamma_1\in \mathbb{R}$ and $\tau>0$, we say that
$a(t,\cdot)$ belongs to {\it $t$-Bloch space} $\mathbb{F}^{\gamma_{1}}_{\tau,\infty}$ if
$a^{\epsilon}_{j,k}(t) := \langle a(t,\cdot), \Phi^{\epsilon}_{j,k}(x)\rangle$ satisfies
$$\sup_{(\epsilon,j,k)\in \Lambda_n}\bigg\{\sup_{t2^{2j\beta}\geq 1}
(t2^{2j\beta})^{\tau} 2^{\frac{nj}{2}} 2^{j\gamma_{1}}|a^{\epsilon}_{j,k}(t)|
+\sup_{0<t2^{2j\beta}\leq 1}
2^{\frac{nj}{2}} 2^{j\gamma_{1}}|a^{\epsilon}_{j,k}(t)|\bigg\}<\infty.$$
We say that $a(t,\cdot)$ belongs to {\it $t$-$L^{\infty}$ space} $\mathbb{F}^{\gamma_{1}}_{0,\infty}$ if
$$\sup_{t> 0} \sup_{j\in \mathbb{Z}, k\in \mathbb{Z}^n}
t^{\frac{-\gamma_{1}}{2\beta}} 2^{\frac{nj}{2}}\big|\big\langle a(t,\cdot),  \Phi^{0}_{j,k}\big\rangle\big|<\infty.$$





The following lemma can be obtained immediately,
\begin{lemma}
Given $1<p, q<\infty$, $\gamma_{1}, \gamma_{2}\in\mathbb{R}$, $m>p$ and $m', \tau>0$.
\begin{itemize}
\item[(i)] If $m>0$, $\mathbb{F}^{\gamma_{1}, \gamma_{2}}_{p,q,m,m'}\subset\mathbb{F}^{\gamma_{1}-\gamma_{2}}_{m, \infty}$.
\item[(ii)] If $-2\beta\tau<\gamma<0<\beta$, then $\mathbb{F}^{\gamma}_{\tau,\infty}\subset \mathbb{F}^{\gamma}_{0,\infty}$, where
$a(t, x)\in\mathbb{F}^{\gamma}_{0,\infty}$ if $t^{-\gamma/2\beta}2^{nj/2}|a^{0}_{j,k}(t)|\lesssim 1$.
\end{itemize}
\end{lemma}
\begin{proof}

We divide the proof into two cases.

Case 1: $t2^{2j\beta}\geq1$. Because $f\in\mathbb{F}^{\gamma_{1},\gamma_{2}, I}_{p,q,m}$,
$$\begin{array}{rl}
&|Q_{r}|^{\frac{\gamma_{2}}{n}-\frac{1}{p}}t^{m}\Big\|\Big(\sum\limits_{j\geq\max\left\{-\log_{2}r, -\frac{\log_{2}t}{2\beta}\right\}}\sum\limits_{Q_{j,k}\subset Q_{r}}2^{qj(\gamma_{1}+\frac{n}{2}+2m\beta)}|a^{\varepsilon}_{j,k}(t)|^{q}\chi(2^{j}x-k)\Big)^{\frac{1}{q}}\Big\|_{p}<\infty.\end{array}$$
Fix $j$, $k$. We have
$$\begin{array}{rl}
&|Q_{r}|^{\frac{\gamma_{1}}{n}-\frac{1}{p}}t^{m}\Big(\int2^{pj(\gamma_{1}+\frac{n}{2}+2m\beta)}|a^{\varepsilon}_{j,k}(t)|^{p}\chi(2^{j}x-k)dx\Big)^{1/p}\\
\lesssim& |Q_{r}|^{\frac{\gamma_{2}}{n}-\frac{1}{p}}t^{m}2^{j(\gamma_{1}+\frac{n}{2}+2m\beta)}|a^{\varepsilon}_{j,k}(t)|2^{-jn/p}.
\end{array}$$
Let $r=2^{-j}$. We can get
$$2^{-j\gamma_{2}}2^{jn/p}2^{j\gamma_{1}}2^{jn/2}(t2^{2j\beta})^{m}|a^{\varepsilon}_{j,k}(t)|2^{-jn/p}\lesssim1,$$
that is,
$$|a^{\varepsilon}_{j,k}(t)|\lesssim 2^{j(\gamma_{2}-\gamma_{1}-\frac{n}{2})}(t2^{2j\beta})^{-m}.$$

Case 2: $0<t2^{2j\beta}<1$. Because $f\in\mathbb{F}^{\gamma_{1},\gamma_{2}, II}_{p,q}$, we can obtain that
$$\begin{array}{rl}
&|Q_{r}|^{\frac{\gamma_{2}}{n}-\frac{1}{p}}\Big\|\Big(\sum\limits_{-\log_{2}r<j<-\frac{\log_{2}t}{2\beta}}\sum\limits_{Q_{j,k}\subset Q_{r}}
2^{qj(\gamma_{1}+\frac{n}{2})}|a^{\varepsilon}_{j,k}(t)|^{q}\chi(2^{j}x-k)\Big)^{\frac{1}{q}}\Big\|_{p}\lesssim 1.
\end{array}$$
This implies that for fixed $\varepsilon$ and $k$,
$$|a^{\varepsilon}_{j,k}(t)|\lesssim 2^{-j(\gamma_{1}-\gamma_{2})}2^{-jn/2}.$$
\end{proof}



\section{Semigroup characterization 
and Riesz operators}\label{sec:semigroup}
In this section, we characterize $\dot{F}^{\gamma_{1}, \gamma_{2}}_{p,q}(\mathbb{R}^{n})$
by Triebel type tent spaces $\mathbb{F}^{\gamma_{1}, \gamma_{2}}_{p, q, m, m'}$.
And we prove the continuity of Riesz operators on these spaces.

\subsection{Semigroup characterization of $\dot{F}^{\gamma_{1}, \gamma_{2}}_{p, q}$
}\label{sec:semi}
For any dyadic cube $Q_{r}$ with side length $r$, we
denote $\widetilde Q_{r}$ the dyadic cube which contains $Q_r$ with
side length $2^{7}r$. For all $ w\in \mathbb{Z}^{n}$, we write
$\widetilde Q_r^w:= 2^{7}rw + \widetilde Q_{r}$ and say
$(\epsilon',k')\in S^{w,j'}_{Q_r}$ if $(\epsilon',j',k')\in
\Lambda_{n}$ and $Q_{j',k'}\subset \widetilde Q_r^w$.   We have

\begin{theorem}\label{thm-cha}
Given $\gamma_{1},\gamma_{2}\in\mathbb{R}$, $1<p<m<\infty$,
$\gamma_{1}-\gamma_{2}<0 <\beta$, $m'>0$ and
$\tau+\frac{\gamma_{1}-\gamma_{2}}{2\beta}>0$.
\begin{itemize}
\item[\rm(i)] If $f \in \dot{F}^{\gamma_{1},\gamma_{2}}_{p,q}$, then
$f*K^{\beta}_{t}\in \mathbb{F}^{\gamma_{1}, \gamma_{2}}_{p, q, m, m'};$
\item[\rm(ii)]  The operator $\pi_{\phi}$ is a bounded and surjective operator from
$\mathbb{F}^{\gamma_{1}, \gamma_{2}}_{p, q, m, m'}$ to
$\dot{F}^{\gamma_{1},\gamma_{2}}_{p,q}$.
\end{itemize}
\end{theorem}
\begin{proof}
(1) We divide the proof into four parts.

{\bf Part I.} $K^{\beta}_{t}\ast f\in\mathbb{F}^{\gamma_{1}, \gamma_{2}, I}_{p, q, m}$.
Because $$|a^{\varepsilon}_{j,k}(t)|\leq e^{-ct2^{2j\beta}}\sum_{\varepsilon', |j-j'|\leq1, k'}\frac{|a^{\varepsilon'}_{j',k'}|}{(1+|2^{j-j'}k'-k|)^{N}},$$
we can get
$$\begin{array}{rl}
&I^{\gamma_{1},\gamma_{2}}_{p,q,m}\\
&=|Q_{r}|^{\gamma_{2}/n-1/p}t^{m}\Big\|\Big(\sum\limits_{j\geq\max\left\{-\log_{2}r, -\frac{\log_{2}t}{2\beta}\right\}}\sum\limits_{Q_{j,k}\subset Q_{r}}2^{qj(\gamma_{1}+\frac{n}{2}+2m\beta)}|a^{\varepsilon}_{j,k}(t)|^{q}\chi(2^{j}x-k)\Big)^{1/q}\Big\|_{p}\\
&\leq|Q_{r}|^{\gamma_{2}/n-1/p}t^{m}\Big\|\Big(\sum\limits_{j\geq\max\left\{-\log_{2}r, -\frac{\log_{2}t}{2\beta}\right\}}\sum\limits_{Q_{j,k}\subset Q_{r}}2^{qj(\gamma_{1}+\frac{n}{2}+2m\beta)}e^{-cqt2^{2j\beta}}\\
&\quad\Big(\sum\limits_{\varepsilon', |j-j'|\leq1, k'}\frac{|a^{\varepsilon'}_{j',k'}|}{(1+|2^{j-j'}k'-k|)^{N}}\Big)^{q}\chi(2^{j}x-k)\Big)^{1/q}\Big\|_{p}\\
&\leq|Q_{r}|^{\gamma_{2}/n-1/p}t^{m}\Big\|\Big(\sum\limits_{j\geq\max\left\{-\log_{2}r, -\frac{\log_{2}t}{2\beta}\right\}}\sum\limits_{Q_{j,k}\subset Q_{r}}2^{qj(\gamma_{1}+\frac{n}{2}+2m\beta)}e^{-cqt2^{2j\beta}}\sum\limits_{|j-j'|\leq1}\\
&\quad\Big(\sum\limits_{\varepsilon', k'}\frac{|a^{\varepsilon'}_{j',k'}|}{(1+|2^{j-j'}k'-k|)^{N}}\Big)^{q}\chi(2^{j}x-k)\Big)^{1/q}\Big\|_{p}\end{array}$$

Let
$$f_{j'}=\sum_{\varepsilon',k'}2^{j'(\gamma_{1}+2m\beta+\frac{n}{2})}|a^{\varepsilon'}_{j',k'}|\chi(2^{j'}x-k').$$
We can get
\begin{eqnarray*}
&&\sum_{\varepsilon',k'}|a^{\varepsilon'}_{j',k'}|(1+|2^{j-j'}k'-k|)^{-N}\\
&=&2^{-j'(\gamma_{1}+\frac{n}{2}+2m\beta)}\sum_{\varepsilon',k'}
\frac{2^{j'(\gamma_{1}+\frac{n}{2}+2m\beta)}|a^{\varepsilon'}_{j',k'}|}{(1+|2^{j-j'}k'-k|)^{N}}\\
&\lesssim&2^{-j'(\gamma_{1}+\frac{n}{2}+2m\beta)}M_{A}(f_{j'})(x),\ x\in Q_{j,k}.
\end{eqnarray*}
Notice that for fixed $j, k$, the number of $Q_{j,k}$ such that $x\in Q_{j,k}$ is 1. Then
$$\begin{array}{rl}
&I^{\gamma_{1},\gamma_{2}}_{p,q,m}\\
&\leq|Q_{r}|^{\gamma_{2}/n-1/p}t^{m}\Big\|\Big(\sum\limits_{j\geq\max\left\{-\log_{2}r, -\frac{\log_{2}t}{2\beta}\right\}}\sum\limits_{Q_{j,k}\subset Q_{r}}2^{qj(\gamma_{1}+\frac{n}{2}+2m\beta)}e^{-cqt2^{2j\beta}}\\
&\quad\sum\limits_{|j-j'|\leq1}2^{-qj'(\gamma_{1}+\frac{n}{2}+2m\beta)}|M_{A}(f_{j'})(x)|^{q}\chi(2^{j}x-k)\Big)^{1/q}\Big\|_{p}\\
&\leq|Q_{r}|^{\gamma_{2}/n-1/p}t^{m}\Big\|\Big(\sum\limits_{j\geq\max\left\{-\log_{2}r, -\frac{\log_{2}t}{2\beta}\right\}}\sum\limits_{|j-j'|\leq1}2^{q(j-j')(\gamma_{1}+\frac{n}{2}+2m\beta)}
|M_{A}(f_{j'})(x)|^{q}\Big)^{1/q}\Big\|_{p}\\
&\leq|Q_{r}|^{\gamma_{2}/n-1/p}t^{m}\Big\|\Big(\sum\limits_{j'\geq\max\left\{-\log_{2}r, -\frac{\log_{2}t}{2\beta}\right\}}|f_{j'}(x)|^{q}\Big)^{1/q}\Big\|_{p}
\end{array}$$

It is easy to see that
$$\begin{array}{rl}
&\Big(\sum\limits_{j'\geq\max\left\{-\log_{2}r, -\frac{\log_{2}t}{2\beta}\right\}}|f_{j'}(x)|^{q}\Big)^{1/q}\\
&\lesssim \Big(\sum\limits_{j'\geq\max\left\{-\log_{2}r, -\frac{\log_{2}t}{2\beta}\right\}}
|\sum\limits_{\varepsilon',k'}2^{j'(\gamma_{1}+2m\beta+\frac{n}{2})}|a^{\varepsilon'}_{j',k'}|\chi(2^{j'}x-k')|^{q}\Big)^{1/q}\\
&\lesssim \Big(\sum\limits_{j'\geq\max\left\{-\log_{2}r, -\frac{\log_{2}t}{2\beta}\right\}}
2^{qj'(\gamma_{1}+2m\beta+\frac{n}{2})}\sum\limits_{\varepsilon',k'}|a^{\varepsilon'}_{j',k'}|^{q}\chi(2^{j'}x-k')|\Big)^{1/q}.
\end{array}$$
Hence we can get $I^{\gamma_{1},\gamma_{2}}_{p,q,m}\lesssim\|f\|_{\dot{F}^{\gamma_{1}, \gamma_{2}}_{p,q}}$.

{\bf Part II.} $K^{\beta}_{t}\ast f\in \mathbb{F}^{\gamma_{1}, \gamma_{2}, II}_{p,q}$.

Similarly we have
$$|a^{\varepsilon}_{j,k}(t)|\lesssim e^{-ct2^{2j\beta}}\sum_{\varepsilon',|j-j'|\leq1, k'}\frac{|a^{\varepsilon'}_{j',k'}|}{(1+|2^{j-j'}k'-k|)^{N}}.$$
Hence
$$\begin{array}{rl}
II^{\gamma_{1}, \gamma_{2}}_{p, q}
&=\ |Q_{r}|^{\frac{\gamma_{2}}{n}-\frac{1}{p}}\Big\|\Big(\sum\limits_{-\log_{2}r<j<-\frac{\log_{2}t}{2\beta}}\sum\limits_{Q_{j,k}\subset Q_{r}}
2^{qj(\gamma_{1}+\frac{n}{2})}|a^{\varepsilon}_{j,k}(t)|^{q}\chi(2^{j}x-k)\Big)^{1/q}\Big\|_{p}\\
&\lesssim\ |Q_{r}|^{\frac{\gamma_{2}}{n}-\frac{1}{p}}\Big\|\Big(\sum\limits_{-\log_{2}r<j<-\frac{\log_{2}t}{2\beta}}\sum\limits_{Q_{j,k}\subset Q_{r}}
2^{qj(\gamma_{1}+\frac{n}{2})}e^{-cqt2^{2j\beta}}\\
&\quad \times\Big(\sum\limits_{\varepsilon',|j-j'|\leq1, k'}\frac{|a^{\varepsilon'}_{j',k'}|}{(1+|2^{j-j'}k'-k|)^{N}}\Big)^{q}\chi(2^{j}x-k)\Big)^{1/q}\Big\|_{p}\\
&\lesssim\ |Q_{r}|^{\frac{\gamma_{2}}{n}-\frac{1}{p}}\sum\limits_{w\in\mathbb{Z}^{n}}(1+|w|)^{-N}\Big\|\Big(\sum\limits_{-\log_{2}r<j<-\frac{\log_{2}t}{2\beta}}\sum\limits_{Q_{j,k}\subset Q_{r}}
2^{qj(\gamma_{1}+\frac{n}{2})}e^{-cqt2^{2j\beta}}\\
&\quad \times\sum\limits_{|j-j'|\leq1}\Big(\sum\limits_{Q_{j',k'}\subset Q^{w}_{j,k}}\frac{|a^{\varepsilon'}_{j',k'}|}{(1+|2^{j-j'}k'-k|)^{N}}\Big)^{q}\chi(2^{j}x-k)\Big)^{1/q}\Big\|_{p}\\
&\lesssim\ |Q_{r}|^{\frac{\gamma_{2}}{n}-\frac{1}{p}}\sum\limits_{w\in\mathbb{Z}^{n}}(1+|w|)^{-N}
\Big\|\Big(\sum\limits_{-\log_{2}r<j<-\frac{\log_{2}t}{2\beta}}\sum\limits_{Q_{j,k}\subset Q_{r}}
2^{qj(\gamma_{1}+\frac{n}{2})}e^{-cqt2^{2j\beta}}\\
&\quad \times\sum\limits_{|j-j'|\leq1}\Big(2^{-j'(\gamma_{1}+\frac{n}{2})}\sum\limits_{Q_{j',k'}\subset Q^{w}_{j,k}}\frac{2^{j'(\gamma_{1}+\frac{n}{2})}|a^{\varepsilon'}_{j',k'}|}{(1+|2^{j-j'}k'-k|)^{N}}\Big)^{q}\chi(2^{j}x-k)\Big)^{1/q}\Big\|_{p}.
\end{array}$$
Let
$$f^{w}_{j'}=\sum_{Q_{j',k'}\subset Q^{w}_{j,k}}2^{j'(\gamma_{1}+\frac{n}{2})}|a^{\varepsilon'}_{j',k'}|\chi(2^{j'}x-k')$$
We can get
$$\begin{array}{rl}
&II^{\gamma_{1}, \gamma_{2}}_{p, q}\\
&\lesssim\quad |Q_{r}|^{\frac{\gamma_{2}}{n}-\frac{1}{p}}\sum\limits_{w\in\mathbb{Z}^{n}}(1+|w|)^{-N}
\Big\|\Big(\sum\limits_{-\log_{2}r<j<-\frac{\log_{2}t}{2\beta}}\sum\limits_{Q_{j,k}\subset Q_{r}}
2^{qj(\gamma_{1}+\frac{n}{2})}\\
&\quad \times\sum\limits_{|j-j'|\leq1}2^{-qj'(\gamma_{1}+\frac{n}{2})}\Big|M_{A}(f^{w}_{j'})(x)\Big|^{q}\chi(2^{j}x-k)\Big)^{1/q}\Big\|_{p}\\
&\lesssim\quad |Q_{r}|^{\frac{\gamma_{2}}{n}-\frac{1}{p}}\sum\limits_{w\in\mathbb{Z}^{n}}(1+|w|)^{-N}
\Big\|\Big(\sum\limits_{-\log_{2}r<j<-\frac{\log_{2}t}{2\beta}}\sum\limits_{Q_{j,k}\subset Q_{r}}
\sum\limits_{|j-j'|\leq1}2^{q(j-j')(\gamma_{1}+\frac{n}{2})}\\
&\quad \times
\Big|\sum\limits_{Q_{j',k'}\subset Q^{w}_{j,k}}2^{j'(\gamma_{1}+\frac{n}{2})}|a^{\varepsilon'}_{j',k'}|\chi(2^{j'}x-k')\Big|^{q}\chi(2^{j}x-k)\Big)^{1/q}\Big\|_{p}.
\end{array}$$
By H\"older's inequality, we obtain
$$\begin{array}{rl}
&\Big(\sum\limits_{Q_{j',k'}\subset Q^{w}_{j,k}}|a^{\varepsilon'}_{j',k'}|\chi(2^{j'}x-k')\Big)^{q}\\
&\lesssim \Big(\sum\limits_{Q_{j',k'}\subset Q^{w}_{j,k}}|a^{\varepsilon'}_{j',k'}|^{q}\chi(2^{j'}x-k')\Big)\Big(\sum\limits_{Q_{j',k'}\subset Q^{w}_{j,k}}1\Big)^{q-1}\\
&\lesssim 2^{n(j'-j)(q-1)}\Big(\sum\limits_{Q_{j',k'}\subset Q^{w}_{j,k}}|a^{\varepsilon'}_{j',k'}|^{q}\chi(2^{j'}x-k')\Big)\\
&\lesssim \Big(\sum\limits_{Q_{j',k'}\subset Q^{w}_{j,k}}|a^{\varepsilon'}_{j',k'}|^{q}\chi(2^{j'}x-k')\Big),
\end{array}$$
where we have used the fact that $|j-j'|\leq1$. By the above estimate, we get
$$\begin{array}{rl}
&II^{\gamma_{1}, \gamma_{2}}_{p, q}\\
&\lesssim\quad |Q_{r}|^{\frac{\gamma_{2}}{n}-\frac{1}{p}}\sum\limits_{w\in\mathbb{Z}^{n}}(1+|w|)^{-N}
\Big\|\Big(\sum\limits_{-\log_{2}r<j<-\frac{\log_{2}t}{2\beta}}\sum\limits_{Q_{j,k}\subset Q_{r}}
\sum\limits_{|j-j'|\leq1}2^{q(j-j')(\gamma_{1}+\frac{n}{2})}\\
&\quad \times
2^{qj'(\gamma_{1}+\frac{n}{2})}\Big(\sum\limits_{Q_{j',k'}\subset Q^{w}_{j,k}}|a^{\varepsilon'}_{j',k'}|^{q}\chi(2^{j'}x-k')\Big)\chi(2^{j}x-k)\Big)^{1/q}\Big\|_{p}\\
&\lesssim \quad |Q_{r}|^{\frac{\gamma_{2}}{n}-\frac{1}{p}}\sum\limits_{w\in\mathbb{Z}^{n}}(1+|w|)^{-N}
\Big\|\Big(\sum\limits_{-\log_{2}r_{w}<j'<-\frac{\log_{2}t}{2\beta}}\sum\limits_{Q_{j',k'}\subset Q_{r_{w}}}\\
&\quad 2^{qj'(\gamma_{1}+\frac{n}{2})}|a^{\varepsilon'}_{j',k'}|^{q}\chi(2^{j'}x-k')\Big)^{1/q}\Big\|_{p}\\
&\lesssim \|f\|_{\dot{F}^{\gamma_{1}, \gamma_{2}}_{p,q}}.
\end{array}$$

{\bf Part III.} $K^{\beta}_{t}\ast f\in \mathbb{F}^{\gamma_{1}, \gamma_{2}, III}_{p,q,m}$.
$$\begin{array}{rl}
III^{\gamma_{1}, \gamma_{2}}_{p,q,m}&=|Q_{r}|^{\frac{\gamma_{2}}{n}-\frac{1}{p}}\Big\|\Big(\sum\limits_{(\varepsilon,j,k)\in\Lambda_{Q}^{n}}
2^{qj(\gamma_{1}+\frac{n}{2}+2m\beta)}\int^{r^{2\beta}}_{2^{-2j\beta}}|a^{\varepsilon}_{j,k}(t)|^{q}t^{qm}\frac{dt}{t}\chi(2^{j}x-k)\Big)^{\frac{1}{q}}\Big\|_{p}\\
&\lesssim |Q_{r}|^{\frac{\gamma_{2}}{n}-\frac{1}{p}}\Big\|\Big(\sum\limits_{(\varepsilon,j,k)\in\Lambda_{Q}^{n}}
2^{qj(\gamma_{1}+\frac{n}{2}+2m\beta)}\\
&\quad \int^{r^{2\beta}}_{2^{-2j\beta}}e^{-cqt2^{2j\beta}}\Big(\sum\limits_{\varepsilon',|j-j'|\leq1, k'}\frac{|a^{\varepsilon'}_{j',k'}|}{(1+|2^{j-j'}k'-k|)^{N}}\Big)^{q}t^{qm}\frac{dt}{t}\chi(2^{j}x-k)\Big)^{\frac{1}{q}}\Big\|_{p}\\
&\lesssim |Q_{r}|^{\frac{\gamma_{2}}{n}-\frac{1}{p}}\sum\limits_{w\in\mathbb{Z}^{n}}(1+|w|)^{-N}\Big\|\Big(\sum\limits_{(\varepsilon,j,k)\in\Lambda_{Q}^{n}}
2^{qj(\gamma_{1}+\frac{n}{2}+2m\beta)}\int^{r^{2\beta}}_{2^{-2j\beta}}e^{-cqt2^{2j\beta}}\\
&\quad \sum\limits_{|j-j'|\leq1}\Big(\sum\limits_{Q_{j' k'}\subset Q_{j,k}^{w}}\frac{|a^{\varepsilon'}_{j',k'}|}{(1+|2^{j-j'}k'-k|)^{N}}\Big)^{q}t^{qm}\frac{dt}{t}\chi(2^{j}x-k)\Big)^{\frac{1}{q}}\Big\|_{p}.
\end{array}$$
It is easy to see that
$$\begin{array}{rl}
\int^{r^{2\beta}}_{2^{-2j\beta}}e^{-cqt2^{2j\beta}}(t2^{2j\beta})^{qm}\frac{dt}{t}\lesssim 1.
\end{array}$$
Then let
$$\begin{array}{rl}
&f^{w}_{j'}(x)=\sum\limits_{Q_{j' k'}\subset Q_{j,k}^{w}}2^{j'(\gamma_{1}+\frac{n}{2})}|a^{\varepsilon'}_{j',k'}|\chi(2^{j'}x-k').
\end{array}$$
We can get
$$\begin{array}{rl}
&III^{\gamma_{1}, \gamma_{2}}_{p,q,m}\\
&\lesssim |Q_{r}|^{\frac{\gamma_{2}}{n}-\frac{1}{p}}\sum\limits_{w\in\mathbb{Z}^{n}}(1+|w|)^{-N}\Big\|\Big(\sum\limits_{(\varepsilon,j,k)\in\Lambda_{Q}^{n}}
2^{qj(\gamma_{1}+\frac{n}{2})}\\
&\quad \sum\limits_{|j-j'|\leq1}\Big(\sum\limits_{Q_{j' k'}\subset Q_{j,k}^{w}}\frac{|a^{\varepsilon'}_{j',k'}|}{(1+|2^{j-j'}k'-k|)^{N}}\Big)^{q}\chi(2^{j}x-k)\Big)^{\frac{1}{q}}\Big\|_{p}\\
&\lesssim |Q_{r}|^{\frac{\gamma_{2}}{n}-\frac{1}{p}}\sum\limits_{w\in\mathbb{Z}^{n}}(1+|w|)^{-N}\Big\|\Big(\sum\limits_{(\varepsilon,j,k)\in\Lambda_{Q}^{n}}
\sum\limits_{|j-j'|\leq1}2^{q(j-j')(\gamma_{1}+\frac{n}{2})}\\
&\quad \times\Big|M_{A}(f_{j'})(x)\Big|^{q}\chi(2^{j}x-k)\Big)^{\frac{1}{q}}\Big\|_{p}\\
&\lesssim |Q_{r}|^{\frac{\gamma_{2}}{n}-\frac{1}{p}}\sum\limits_{w\in\mathbb{Z}^{n}}\frac{1}{(1+|w|)^{N}}
\Big\|\Big(\sum\limits_{j\geq-\log_{2}r}\sum\limits_{Q_{j,k}\subset Q_{r}}
\sum\limits_{|j-j'|\leq1}2^{q(j-j')(\gamma_{1}+\frac{n}{2})}\\
&\quad \times\Big|\sum\limits_{Q_{j' k'}\subset Q_{j,k}^{w}}2^{j'(\gamma_{1}+\frac{n}{2})}|a^{\varepsilon'}_{j',k'}|\chi(2^{j'}x-k')\Big|^{q}\chi(2^{j}x-k)\Big)^{\frac{1}{q}}\Big\|_{p}\\
&\lesssim |Q_{r}|^{\frac{\gamma_{2}}{n}-\frac{1}{p}}\sum\limits_{w\in\mathbb{Z}^{n}}\frac{1}{(1+|w|)^{N}}
\Big\|\Big(\sum\limits_{j'\geq-\log_{2}r_{w}}
\sum\limits_{Q_{j',k'}\subset Q^{w}_{r}}2^{qj'(\gamma_{1}+\frac{n}{2})}|a^{\varepsilon'}_{j',k'}|^{q}\chi(2^{j'}x-k')\Big)^{\frac{1}{q}}\Big\|_{p}\\
&\lesssim \|f\|_{\dot{F}^{\gamma_{1}, \gamma_{2}}_{p,q}}.
\end{array}$$

{\bf Part IV.} $K^{\beta}_{t}\ast f\in\mathbb{F}^{\gamma_{1}, \gamma_{2}, IV}_{p,q, m'}$.

Similar to the proof of Part III, we can obtain that
$$\begin{array}{rl}
&IV^{\gamma_{1}, \gamma_{2}}_{p,q,m'}\\
&=|Q_{r}|^{\frac{\gamma_{2}}{n}-\frac{1}{p}}\Big\|\Big(\sum\limits_{(\varepsilon,j,k)\in\Lambda^{n}_{Q}}
2^{qj(\gamma_{1}+\frac{n}{2}+2m'\beta)}\int^{2^{-2j\beta}}_{0}t^{qm'}e^{-cqt2^{2j\beta}}\\
&\quad\Big(\sum\limits_{|j-j'|\leq1}\sum\limits_{\varepsilon',k'}\frac{|a^{\varepsilon'}_{j',k'}|}{(1+|2^{j-j'}k'-k|^{N})}\Big)^{q}
\frac{dt}{t}\chi(2^{j}x-k)\Big)^{1/q}\Big\|_{p}\\
&\lesssim |Q_{r}|^{\frac{\gamma_{2}}{n}-\frac{1}{p}}\Big\|\Big(\sum\limits_{(\varepsilon,j,k)\in\Lambda^{n}_{Q}}
2^{qj(\gamma_{1}+\frac{n}{2})}\\
&\quad\Big(\sum\limits_{|j-j'|\leq1}\sum\limits_{\varepsilon',k'}\frac{|a^{\varepsilon'}_{j',k'}|}{(1+|2^{j-j'}k'-k|^{N})}\Big)^{q}
\chi(2^{j}x-k)\Big)^{1/q}\Big\|_{p}\\
&\lesssim |Q_{r}|^{\frac{\gamma_{2}}{n}-\frac{1}{p}}\sum\limits_{w\in\mathbb{Z}^{n}}(1+|w|)^{-N}\Big\|\Big(\sum\limits_{(\varepsilon,j,k)\in\Lambda_{Q}^{n}}
\sum\limits_{|j-j'|\leq1}2^{q(j-j')(\gamma_{1}+\frac{n}{2})}\\
&\quad \times\Big|M_{A}(f_{j'})(x)\Big|^{q}\chi(2^{j}x-k)\Big)^{\frac{1}{q}}\Big\|_{p}\\
&\lesssim |Q_{r}|^{\frac{\gamma_{2}}{n}-\frac{1}{p}}\sum\limits_{w\in\mathbb{Z}^{n}}\frac{1}{(1+|w|)^{N}}
\Big\|\Big(\sum\limits_{j'\geq-\log_{2}r_{w}}
\sum\limits_{Q_{j',k'}\subset Q^{w}_{r}}2^{qj'(\gamma_{1}+\frac{n}{2})}|a^{\varepsilon'}_{j',k'}|^{q}\chi(2^{j'}x-k')\Big)^{\frac{1}{q}}\Big\|_{p}\\
&\lesssim \|f\|_{\dot{F}^{\gamma_{1}, \gamma_{2}}_{p,q}}.
\end{array}$$

(2) Now we prove
$$\pi_{\Phi}f(t,x):\ \mathbb{F}^{\gamma_{1}, \gamma_{2}}_{p,q,m,m'}\rightarrow \dot{F}^{\gamma_{1}, \gamma_{2}}_{p,q}.$$
For the dyadic cube $Q_{r}$ and $j\geq-\log_{2}r$, H\'older's inequality implies that
$$\begin{array}{rl}
|a^{\varepsilon}_{j,k}|^{q}
&\lesssim\Big|\sum\limits_{\varepsilon', |j-j'|\leq1, k'}\int^{\infty}_{0}\Big\{(\max\{t2^{2j'\beta}, t^{-1}2^{-2j'\beta}\})^{-N}\\
&\quad \times|a^{\varepsilon'}_{j',k'}(t)|(1+|2^{j-j'}k'-k|^{-N})\Big\}\frac{dt}{t}\Big|^{q}\\
&\lesssim \sum\limits_{|j-j'|\leq1}\sum\limits_{\varepsilon',k'}(1+|2^{j'-j}k'-k|)^{-N}
\Big(\int^{\infty}_{r^{2\beta}}(t2^{2j'\beta})^{-N}|a^{\varepsilon'}_{j',k'}(t)|\frac{dt}{t}\Big)^{q}\\
&+ \sum\limits_{|j-j'|\leq1}\sum\limits_{\varepsilon',k'}(1+|2^{j'-j}k'-k|)^{-N}
\Big(\int_{2^{-2j'\beta}}^{r^{2\beta}}(t2^{2j'\beta})^{-N}|a^{\varepsilon'}_{j',k'}(t)|\frac{dt}{t}\Big)^{q}\\
&+ \sum\limits_{|j-j'|\leq1}\sum\limits_{\varepsilon',k'}(1+|2^{j'-j}k'-k|)^{-N}
\Big(\int^{2^{-2j'\beta}}_{0}(t2^{2j'\beta})^{N}|a^{\varepsilon'}_{j',k'}(t)|\frac{dt}{t}\Big)^{q}.
\end{array}$$
By
$$\begin{array}{rl}
(\pi_{\Phi}f(t, \cdot))&=\int^{\infty}_{0}f(t, \cdot)\ast\Phi^{\beta}_{t}(x)\frac{dt}{t}\\
&=:\sum\limits_{(\varepsilon,j,k)\in\Lambda_{n}}a^{\varepsilon}_{j,k}\Phi^{\varepsilon}_{j,k}(x).
\end{array}$$
We estimate the term
$$\begin{array}{rl}
M&=|Q_{r}|^{\frac{\gamma_{2}}{n}-\frac{1}{p}}\Big\|\Big(\sum\limits_{(\varepsilon,j,k)\in \Lambda^{n}_{Q}}2^{qj(\gamma_{1}+\frac{n}{2})}|a^{\varepsilon}_{j,k}|^{q}\chi(2^{j}x-k)\Big)^{1/q}\Big\|_{L^{p}}\\
&\lesssim M_{1}+M_{2}+M_{3}.\end{array}$$

For $M_{1}$, we get
$$\begin{array}{rl}
M_{1}&\lesssim |Q_{r}|^{\frac{\gamma_{2}}{n}-\frac{1}{p}}\Big\|\Big(\sum\limits_{(\varepsilon,j,k)\in \Lambda^{n}_{Q}}2^{qj(\gamma_{1}+\frac{n}{2})}\sum\limits_{|j-j'|\leq1}\sum\limits_{\varepsilon',k'}(1+|2^{j-j'}k'-k|)^{-N}\\
&\quad \Big(\int^{\infty}_{r^{2\beta}}(t2^{2j'\beta})^{-N}|a^{\varepsilon'}_{j',k'}(t)|\frac{dt}{t}\Big)^{q}\chi(2^{j}x-k)\Big)^{1/q}\Big\|_{p}.
\end{array}$$
Because $f(t, x)\in\mathbb{F}^{\gamma_{1},\gamma_{2}}_{p,q,m,m'}\subset \mathbb{F}^{\gamma_{1}-\gamma_{2}}_{p,\infty}$, if $t2^{2j'\beta}\geq1$,
$$|a^{\varepsilon'}_{j',k'}(t)|\lesssim (t2^{2j'\beta})^{-m}2^{-nj'/2}2^{-j'(\gamma_{1}-\gamma_{2})}\lesssim 2^{-nj'/2}2^{-j'(\gamma_{1}-\gamma_{2})}.$$
We can get
$$\begin{array}{rl}
M_{1}&\lesssim |Q_{r}|^{\frac{\gamma_{2}}{n}-\frac{1}{p}}\Big\|\Big(\sum\limits_{(\varepsilon,j,k)\in\Lambda^{n}_{Q}}2^{qj(\gamma_{1}+\frac{n}{2})}\sum\limits_{|j-j'|\leq1}
\sum\limits_{\varepsilon',k'}(1+|2^{j-j'}k'-k|)^{-N}\\
&\quad \Big(\int^{\infty}_{r^{2\beta}}(t2^{2j'\beta})^{-N}2^{-nj'/2}2^{-j(\gamma_{1}-\gamma_{2})}\frac{dt}{t}\Big)^{q}\chi(2^{j}x-k)\Big)^{1/q}\Big\|_{p}\\
&\lesssim |Q_{r}|^{\frac{\gamma_{2}}{n}-\frac{1}{p}}\Big\|\Big(\sum\limits_{(\varepsilon,j,k)\in\Lambda_{Q}^{n}}2^{qj(\gamma_{1}+\frac{n}{2})}\sum\limits_{|j-j'|\leq1}
\sum\limits_{\varepsilon',k'}(1+|2^{j-j'}k'-k|)^{-N}\\
&\quad [2^{-nj'/2}2^{-j'(\gamma_{1}-\gamma_{2})}(r2^{j'})^{-2\beta N}]^{q}\chi(2^{j}x-k)\Big)^{1/q}\Big\|_{p}\\
&\lesssim |Q_{r}|^{\frac{\gamma_{2}}{n}-\frac{1}{p}}\Big\|\Big(\sum\limits_{(\varepsilon,j,k)\in\Lambda_{Q}^{n}}2^{qj(\gamma_{1}+\frac{n}{2})}\sum\limits_{|j-j'|\leq1}
\sum\limits_{\varepsilon',k'}(1+|2^{j-j'}k'-k|)^{-N}\\
&\quad 2^{-qnj'/2}2^{-qj'(\gamma_{1}-\gamma_{2})}(r2^{j'})^{-2q\beta N}\chi(2^{j}x-k)\Big)^{1/q}\Big\|_{p}\\
&\lesssim \Big\|\Big(\sum\limits_{(\varepsilon,j,k)\in\Lambda_{Q}^{n}}2^{qj(\gamma_{1}+\frac{n}{2})}\sum\limits_{|j-j'|\leq1}
\sum\limits_{w\in\mathbb{Z}^{n}}\sum\limits_{Q_{j',k'\subset Q^{w}_{j,k}}}(1+|2^{j-j'}k'-k|)^{-N}\\
&\quad 2^{-qnj'/2}2^{-qj'(\gamma_{1}-\gamma_{2})}(r2^{j'})^{-2q\beta N}\chi(2^{j}x-k)\Big)^{1/q}\Big\|_{p}\\
&\lesssim |Q_{r}|^{\frac{\gamma_{2}}{n}-\frac{1}{p}}\sum\limits_{j\geq-\log_{2}r}2^{j(\gamma_{1}+\frac{n}{2})}\sum\limits_{(\varepsilon,k)\in S^{j}_{Q_{r}}}
\sum\limits_{|j-j'|\leq1}\sum\limits_{w\in\mathbb{Z}^{n}}(1+|w|)^{-N}\\
&\quad \sum\limits_{Q_{j',k'}\subset Q^{w}_{j,k}}2^{-nj'/2}2^{-j'(\gamma_{1}-\gamma_{2})}(r2^{j'})^{-2\beta N}\|\chi(2^{j}x-k)\|_{p}\\
&\lesssim |Q_{r}|^{\frac{\gamma_{2}}{n}-\frac{1}{p}}\sum\limits_{j'\geq-\log_{2}r-1}2^{j'(\gamma_{1}+\frac{n}{2})}2^{-nj'/2}2^{-j'(\gamma_{1}-\gamma_{2})}
(r2^{j'})^{-2\beta N}(r2^{j'})^{n}2^{-nj'/p}\\
&\lesssim 1.
\end{array}$$
For $M_{2}$, we can get
$$\begin{array}{rl}
M_{2}&\lesssim|Q_{r}|^{\frac{\gamma_{2}}{n}-\frac{1}{p}}\Big\|\Big(\sum\limits_{(\varepsilon,j,k)\in\Lambda_{Q}^{n}}2^{qj(\gamma_{1}+\frac{n}{2})}
\Big|\sum\limits_{|j-j'|\leq1}\sum\limits_{\varepsilon',k'}(1+|2^{j-j'}k'-k|)^{-N}\\
&\quad \Big(\int^{r^{2\beta}}_{2^{-2j'\beta-1}}(t2^{2j'\beta})^{-N}|a^{\varepsilon'}_{j',k'}(t)|\frac{dt}{t}\Big)\Big|^{q}\chi(2^{j}x-k)\Big)^{1/q}\Big\|_{p}.
\end{array}$$

Let
$$\begin{array}{rl}
b^{\varepsilon'}_{j',k'}&=\int^{r^{2\beta}}_{2^{-1-2j'\beta}}(t2^{2j'\beta})^{-N}|a^{\varepsilon'}_{j',k'}(t)|\frac{dt}{t}.
\end{array}$$
We can get
$$\begin{array}{rl}
&\sum\limits_{|j-j'|\leq1}\sum\limits_{\varepsilon',k'}(1+|2^{j-j'}k'-k|)^{-N}\Big(\int^{r^{2\beta}}_{2^{-2j'\beta}}
(t2^{2j'\beta})^{-N}|a^{\varepsilon'}_{j',k'}(t)|\frac{dt}{t}\Big)\\
=:&\sum\limits_{|j-j'|\leq1}\sum\limits_{\varepsilon',k'}(1+|2^{j-j'}k'-k|)^{-N}b^{\varepsilon'}_{j',k'}\\
\lesssim&2^{-j'(\gamma_{1}+2m\beta+\frac{n}{2})}\sum\limits_{\varepsilon',k'}
\frac{2^{j'(\gamma_{1}+2m\beta+\frac{n}{2})}|b^{\varepsilon'}_{j',k'}|}{(1+|2^{j-j'}k'-k|)^{N}}\\
\lesssim&2^{-j'(\gamma_{1}+2m\beta+\frac{n}{2})}M_{A}(f_{j'})(x),\ x\in Q_{j,k},
\end{array}$$
where $f_{j'}(x)=\sum\limits_{\varepsilon',k'}2^{j'(\gamma_{1}+2m\beta+\frac{n}{2})}|b^{\varepsilon'}_{j',k'}|\chi(2^{j'}x-k')$.
Then
$$\begin{array}{rl}
M_{2}
&\lesssim |Q_{r}|^{\frac{\gamma_{2}}{n}-\frac{1}{p}}\Big\|\Big(\sum\limits_{(\varepsilon,j,k)\in\Lambda^{n}_{Q}}\sum\limits_{|j-j'|\leq1}
2^{qj(\gamma_{1}+\frac{n}{2})}2^{-qj'(\gamma_{1}+\frac{n}{2})}\\
&\quad\times \Big|M_{A}(f_{j'})(x)\Big|^{q}\chi(2^{j}x-k)\Big)^{1/q}\Big\|_{p}\\
&\lesssim |Q_{r}|^{\frac{\gamma_{2}}{n}-\frac{1}{p}}\sum\limits_{w\in\mathbb{Z}^{n}}(1+|w|)^{-N}\Big\|\Big(\sum\limits_{j'\geq-\log_{2}r_{w}}
\sum\limits_{Q_{j',k'}\subset Q^{w}_{r}}2^{qj'(\gamma_{1}+\frac{n}{2})}\\
&\quad\times|b^{\varepsilon'}_{j',k'}|^{q}\chi(2^{j'}x-k')\Big)^{1/q}\Big\|_{p}.
\end{array}$$
By H\"older's inequality, we can see that
$$\begin{array}{rl}
|b^{\varepsilon'}_{j',k'}|^{q}&=\Big(\int^{r^{2\beta}}_{2^{}-2j'\beta}(t2^{2j'\beta})^{-N}|a^{\varepsilon'}_{j',k'}(t)|\frac{dt}{t}\Big)^{q}\\
&\lesssim\Big(\int^{r^{2\beta}}_{2^{}-2j'\beta}|a^{\varepsilon'}_{j',k'}(t)|^{q}(t2^{2j'\beta})^{qm}\frac{dt}{t}\Big)
\Big(\int^{r^{2\beta}}_{2^{}-2j'\beta}(t2^{2j'\beta})^{-qm-N}\frac{dt}{t}\Big)^{q-1}\\
&\lesssim \int^{r^{2\beta}}_{2^{}-2j'\beta}|a^{\varepsilon'}_{j',k'}(t)|^{q}(t2^{2j'\beta})^{qm}\frac{dt}{t}.
\end{array}$$

Finally we obtain
$$\begin{array}{rl}
M_{2}&\lesssim |Q_{r}|^{\frac{\gamma_{2}}{n}-\frac{1}{p}}\sum\limits_{w\in\mathbb{Z}^{n}}\frac{1}{(1+|w|)^{N}}\Big\|\Big(\sum\limits_{j'\geq-\log_{2}r_{w}}
2^{qj'(\gamma_{1}+\frac{n}{2}+2m\beta)}\\
&\quad \sum\limits_{Q_{j',k'}\subset Q^{w}_{r}}\int^{r^{2\beta}}_{2^{}-2j'\beta}|a^{\varepsilon'}_{j',k'}(t)|^{q}(t2^{2j'\beta})^{qm}\frac{dt}{t}
\chi(2^{j'}x-k')\Big)^{1/q}\Big\|_{p}\\
&\lesssim \|f\|_{\mathbb{F}^{\gamma_{1}, \gamma_{2}, III}_{p,q, m}}.
\end{array}$$

Now we deal with $M_{3}$.
$$\begin{array}{rl}
M_{3}&\lesssim |Q_{r}|^{\frac{\gamma_{2}}{n}-\frac{1}{p}}\Big\|\Big(\sum\limits_{(\varepsilon,j,k)\in\Lambda_{Q}^{n}}2^{qj(\gamma_{1}+\frac{n}{2})}
\Big|\sum\limits_{|j-j'|\leq1}\sum\limits_{\varepsilon',k'}(1+|2^{j-j'}k'-k|)^{-N}\\
&\quad\Big(\int^{2^{-2j'\beta}}_{0}(t2^{2j'\beta})^{N}|a^{\varepsilon'}_{j',k'}(t)|\frac{dt}{t}\Big)\Big|^{q}\chi(2^{j}x-k)\Big)^{1/q}\Big\|_{p}.
\end{array}$$

Let
$$\begin{array}{rl}
b^{\varepsilon'}_{j',k'}=\int^{2^{-2j'\beta}}_{0}(t2^{2j'\beta})^{N}|a^{\varepsilon'}_{j',k'}(t)|\frac{dt}{t}.
\end{array}$$
We obtain that
$$\begin{array}{rl}
&\sum\limits_{\varepsilon', k'}(1+|2^{j'-j}k'-k|)^{-N}\Big(\int^{2^{-2j'\beta}}_{0}(t2^{2j'\beta})^{N}|a^{\varepsilon'}_{j',k'}(t)|\frac{dt}{t}\Big)\\
\lesssim&2^{-j'(\gamma_{1}+\frac{n}{2})}\sum\limits_{\varepsilon',k'}\frac{2^{j'(\gamma_{1}+\frac{n}{2})}|b^{\varepsilon'}_{j',k'}|}{(1+|2^{j'-j}k'-k|)^{N}}\\
\lesssim&2^{-j'(\gamma_{1}+\frac{n}{2})}M_{A}(f_{j'})(x),\ x\in Q_{j,k},
\end{array}$$
where
$$\begin{array}{rl}
f_{j'}=\sum\limits_{\varepsilon',k'}2^{j'(\gamma_{1}+\frac{n}{2})}|b^{\varepsilon'}_{j',k'}|\chi(2^{j'}x-k').
\end{array}$$
So
$$\begin{array}{rl}
M_{3}&\lesssim |Q_{r}|^{\frac{\gamma_{2}}{n}-\frac{1}{p}}\sum\limits_{w\in\mathbb{Z}^{n}}(1+|w|)^{-N}\\
&\quad\Big\|\Big(\sum\limits_{j\geq-\log_{2}r}2^{qj'(\gamma_{1}+\frac{n}{2})}
\sum\limits_{Q_{\varepsilon,k'}\subset Q^{w}_{r}}|b^{\varepsilon'}_{j',k'}|^{q}\chi(2^{j'}x-k')\Big)^{1/q}\Big\|_{p}.
\end{array}$$
By H\"older's inequality, we get
$$\begin{array}{rl}
|b^{\varepsilon'}_{j',k'}|=&\Big(\int^{2^{-2j'\beta}}_{0}(t2^{2j'\beta})^{N}|a^{\varepsilon'}_{j',k'}(t)|\frac{dt}{t}\Big)^{q}\\
\lesssim&\Big(\int^{2^{-2j'\beta}}_{0}(t2^{2j'\beta})^{qm'}|a^{\varepsilon'}_{j',k'}(t)|^{q}\frac{dt}{t}\Big)
\Big(\int^{2^{-2j'\beta}}_{0}(t2^{2j'\beta})^{N-qm'}\frac{dt}{t}\Big)^{q-1}\\
\lesssim&\int^{2^{-2j'\beta}}_{0}(t2^{2j'\beta})^{qm'}|a^{\varepsilon'}_{j',k'}(t)|^{q}\frac{dt}{t}
\end{array}$$
Hence
$$\begin{array}{rl}
M_{3}\lesssim&|Q_{r}|^{\frac{\gamma_{2}}{n}-\frac{1}{p}}\sum\limits_{w\in\mathbb{Z}^{n}}(1+|w|)^{-N}\\
&\quad\Big\|\Big(\sum\limits_{j'\geq-\log_{2}r_{w}}2^{qj'(\gamma_{1}+\frac{n}{2}+2m'\beta)}\int^{2^{-2j'\beta}}_{0}t^{qm'}|a^{\varepsilon'}_{j',k'}(t)|^{q}
\frac{dt}{t}\Big)^{1/q}\Big\|_{p}\\
\lesssim&\|f\|_{\mathbb{F}^{\gamma_{1},\gamma_{2}, IV}_{p,q,m'}}.
\end{array}$$
\end{proof}

\subsection{Continuity of Riesz operators on $\mathbb{F}^{\gamma_{1}, \gamma_{2}}_{p,q,m,m'}$}\label{sec:Riesz}
For Riesz operators $R_{l}$ $(l=1,\cdots,n)$,
$\forall (\epsilon,j,k), (\epsilon',j',k')\in \Lambda_{n}$, denote
$$a^{\epsilon,\epsilon',l}_{j,k,j',k'}= \langle \Phi^{\epsilon}_{j,k},
R_l \Phi^{\epsilon'}_{j',k'}\rangle.$$ If $|j-j'|\geq 2$, then
$a^{\epsilon,\epsilon',l}_{j,k,j',k'}=0$.
Similar to the proof in the Lemma \ref{lem:CZcontinuity},  we can verify  that the Riesz
transforms $R_{l}$, $l=1,\cdots,n$, are continuous on $\mathbb{F}^{\gamma_{1}, \gamma_{2}}_{p,q,m,m'}$. See also
\cite{Al, MY, Yang1}.
\begin{theorem}
Given $1<p, q<\infty$, $\gamma_{1}$, $\gamma_{2}\in\mathbb{R}$, $m>p$ and $m'>0$. The Riesz transforms $R_{1}, R_{2}\cdots, R_{n}$ are bounded on $\mathbb{F}^{\gamma_{1}, \gamma_{2}}_{p,q,m,m'}$.
\end{theorem}
\begin{proof}
$$\begin{array}{rl}
R_{l}g(t, x)&=\sum\limits_{(\varepsilon,j,k)\in\Lambda_{n}}g^{\varepsilon}_{j,k}(t)R_{l}\Phi^{\varepsilon}_{j,k}(x)\\
&=:\sum\limits_{(\varepsilon,j,k)\in\Lambda_{n}}b^{\varepsilon}_{j,k}(t)\Phi^{\varepsilon}_{j,k}(x),
\end{array}$$
where
$$\begin{array}{rl}
b^{\varepsilon}_{j,k}(t)=&\langle R_{l}g(t, \cdot),\ \Phi^{\varepsilon}_{j,k}\rangle\\
=&\sum\limits_{|j-j'|\leq1}\sum\limits_{\varepsilon',k'}a^{\varepsilon,\varepsilon'}_{j,k,j',k'}g^{\varepsilon'}_{j',k'}(t).
\end{array}$$
By Lemma \ref{lem:CZO}, we have
$$\begin{array}{rl}
|a^{\varepsilon,\varepsilon'}_{j,k,j',k'}|\lesssim 2^{-|j-j'|(\frac{n}{2}+N_{0})}\Big(\frac{2^{-j}+2^{-j'}}{2^{-j}+2^{-j'}+|2^{-j}k-2^{-j'}k'|}\Big)^{n+N_{0}}.
\end{array}$$
We divide the proof into four parts.

{\bf Step I:} $(R_{l}g)(t, x)\in\mathbb{F}^{\gamma_{1}, \gamma_{2}, I}_{p, q, m}$, $l=1, 2, \cdots, n$.

We can see that
$$\begin{array}{rl}
I^{\gamma_{1}, \gamma_{2}}_{p, q, m}=&|Q|^{\frac{\gamma_{2}}{n}-\frac{1}{p}}t^{m}\Big\|\Big(\sum\limits_{j\geq\max\left\{-\log_{2}r, -\frac{\log_{2}t}{2\beta}\right\}}\sum\limits_{Q_{j,k}\subset Q_{r}}2^{qj(\gamma_{1}+\frac{n}{2}+2m\beta)}\\
&\quad |b^{\varepsilon}_{j,k}(t)|^{q}\chi(2^{j}x-k)\Big)^{1/q}\Big\|_{p}\\
\lesssim&|Q|^{\frac{\gamma_{2}}{n}-\frac{1}{p}}t^{m}\Big\|\Big(\sum\limits_{j\geq\max\left\{-\log_{2}r, -\frac{\log_{2}t}{2\beta}\right\}}\sum\limits_{Q_{j,k}\subset Q_{r}}2^{qj(\gamma_{1}+\frac{n}{2}+2m\beta)}\\
&\quad\Big|\sum\limits_{|j-j'|\leq1}\sum\limits_{\varepsilon',k'}a^{\varepsilon,\varepsilon'}_{j,k,j',k'}g^{\varepsilon'}_{j',k'}(t)\Big|^{q}
\chi(2^{j}x-k)\Big)^{1/q}\Big\|_{p}\\
\lesssim&|Q|^{\frac{\gamma_{2}}{n}-\frac{1}{p}}t^{m}\Big\|\Big(\sum\limits_{j\geq\max\left\{-\log_{2}r, -\frac{\log_{2}t}{2\beta}\right\}}\sum\limits_{Q_{j,k}\subset Q_{r}}2^{qj(\gamma_{1}+\frac{n}{2}+2m\beta)}\\
&\quad\sum\limits_{|j-j'|\leq1}\Big|\sum\limits_{\varepsilon',k'}\frac{g^{\varepsilon'}_{j',k'}(t)}{(1+|2^{j-j'}k'-k|)^{n+N_{0}}}\Big|^{q}
\chi(2^{j}x-k)\Big)^{1/q}\Big\|_{p}\\
\end{array}$$
Let
$$\begin{array}{rl}
f_{j'}=\sum\limits_{\varepsilon',k'}2^{j'(\gamma_{1}+\frac{n}{2}+2m\beta)}|g^{\varepsilon'}_{j',k'}(t)|\chi(2^{j'}x-k').
\end{array}$$
Then we get
$$\begin{array}{rl}
I^{\gamma_{1}, \gamma_{2}}_{p,q,m}(t)\lesssim&|Q_{r}|^{\frac{\gamma_{2}}{n}-\frac{1}{p}}t^{m}\Big\|\Big(\sum\limits_{j\geq\max\left\{-\log_{2}r, -\frac{\log_{2}t}{2\beta}\right\}}\sum\limits_{Q_{j,k}\subset Q_{r}}2^{qj(\gamma_{1}+\frac{n}{2}+2m\beta)}\\
&\quad\sum\limits_{|j-j'|\leq1}2^{-qj'(\gamma_{1}+\frac{n}{2}+2m\beta)}|M_{A}(f_{j'})(x)|^{q}\chi(2^{j}x-k)\Big)^{1/q}\Big\|_{p}\\
\lesssim&|Q_{r}|^{\frac{\gamma_{2}}{n}-\frac{1}{p}}t^{m}\Big\|\Big(\sum\limits_{j\geq\max\left\{-\log_{2}r, -\frac{\log_{2}t}{2\beta}\right\}}\sum\limits_{Q_{j,k}\subset Q_{r}}\\
&\quad\sum\limits_{|j-j'|\leq1}2^{q(j-j')(\gamma_{1}+\frac{n}{2}+2m\beta)}|M_{A}(f_{j'})(x)|^{q}\chi(2^{j}x-k)\Big)^{1/q}\Big\|_{p}\\
\lesssim&|Q_{r}|^{\frac{\gamma_{2}}{n}-\frac{1}{p}}t^{m}\sum\limits_{w\in\mathbb{Z}^{n}}\frac{1}{(1+|w|)^{N}}\Big\|\Big(\sum\limits_{j'\geq\max\left\{-\log_{2}r_{w}, -\frac{-\log_{2}t}{2\beta}\right\}}
|f_{j'}(x)|^{q}\Big)^{1/q}\Big\|_{p}\\
\lesssim&|Q_{r}|^{\frac{\gamma_{2}}{n}-\frac{1}{p}}t^{m}\sum\limits_{w\in\mathbb{Z}^{n}}\frac{1}{(1+|w|)^{N}}\Big\|\Big(\sum\limits_{j'\geq\max\left\{-\log_{2}r_{w}, -\frac{-\log_{2}t}{2\beta}\right\}}2^{qj'(\gamma_{1}+\frac{n}{2}+2m\beta)}\\
&\quad\times\Big(\sum\limits_{Q_{j',k'}\subset Q_{r}^{w}}|g^{\varepsilon'}_{j',k'}(t)|\chi(2^{j'}x-k')\Big)^{q}\Big)^{1/q}\Big\|_{p}.
\end{array}$$
By H\"older's inequality, we obtain
$$\begin{array}{rl}
I^{\gamma_{1}, \gamma_{2}}_{p,q,m}(t)\lesssim&|Q_{r}|^{\frac{\gamma_{2}}{n}-\frac{1}{p}}t^{m}
\sum\limits_{w\in\mathbb{Z}^{n}}\frac{1}{(1+|w|)^{N}}\Big\|\Big(\sum\limits_{j'\geq\max\left\{-\log_{2}r_{w}, -\frac{-\log_{2}t}{2\beta}\right\}}2^{qj'(\gamma_{1}+\frac{n}{2}+2m\beta)}\\
&\quad\times\sum\limits_{Q_{j',k'}\subset Q_{r}^{w}}|g^{\varepsilon'}_{j',k'}(t)|^{q}\chi(2^{j'}x-k')\Big)^{1/q}\Big\|_{p}\\
\lesssim&\|g\|_{\mathbb{F}^{\gamma_{1}, \gamma_{2}, I}_{p,q,m}}+\|g\|_{\mathbb{F}^{\gamma_{1}, \gamma_{2}, II}_{p,q}}
\end{array}$$

{\bf Step II:} $(R_{l}g)(t, x)\in\mathbb{F}^{\gamma_{1}, \gamma_{2}, II}_{p, q}$, $l=1, 2, \cdots, n$.

For this term, let the radius $r$ of $Q_{r}$ be $2^{-j_{0}}$.

$$\begin{array}{rl}
II^{\gamma_{1}, \gamma_{2}}_{p,q, Q_{r}}=&|Q_{r}|^{\frac{\gamma_{2}}{n}-\frac{1}{p}}\Big\|\Big(\sum\limits_{j_{0}<j<-\frac{\log_{2}t}{2\beta}}\sum\limits_{Q_{j,k}\subset Q_{r}}2^{qj(\gamma_{1}+\frac{n}{2})}|a^{\varepsilon}_{j,k}(t)|^{q}\chi(2^{j}x-k)\Big)^{1/q}\Big\|_{p}\\
\lesssim&|Q_{r}|^{\frac{\gamma_{2}}{n}-\frac{1}{p}}\Big\|\Big(\sum\limits_{j_{0}<j<-\frac{\log_{2}t}{2\beta}}\sum\limits_{Q_{j,k}\subset Q_{r}}2^{qj(\gamma_{1}+\frac{n}{2})}\\
&\quad|\sum\limits_{|j-j'|\leq1}
\sum\limits_{\varepsilon',k'}a^{\varepsilon,\varepsilon'}_{j,k,j',k'}g^{\varepsilon'}_{j',k'}(t)|^{q}\chi(2^{j}x-k)\Big)^{1/q}\Big\|_{p}\\
\lesssim&|Q_{r}|^{\frac{\gamma_{2}}{n}-\frac{1}{p}}\Big\|\Big(\sum\limits_{j_{0}<j<-\log_{2}t/2\beta}\sum\limits_{Q_{j,k}\subset Q_{r}}2^{qj(\gamma_{1}+\frac{n}{2})}\\
&\quad\sum\limits_{|j-j'|\leq1}|
\sum\limits_{\varepsilon',k'}\frac{g^{\varepsilon'}_{j',k'}(t)}{(1+|2^{j-j'}k'-k|)^{n+N_{0}}}|^{q}\chi(2^{j}x-k)\Big)^{1/q}\Big\|_{p}.
\end{array}$$
Let
$$\begin{array}{rl}
f_{j'}=2^{j'(\gamma_{1}+\frac{n}{2})}\sum\limits_{Q_{j',k'}\subset Q^{w}_{j,k}}|g^{\varepsilon'}_{j',k'}(t)|\chi(2^{j'}x-k').
\end{array}$$
We have
$$\begin{array}{rl}
II^{\gamma_{1}, \gamma_{2}}_{p,q, Q_{r}}\lesssim&|Q_{r}|^{\frac{\gamma_{2}}{n}-\frac{1}{p}}\sum\limits_{w\in\mathbb{Z}^{n}}(1+|w|)^{-N}\Big\|\Big(\sum\limits_{j_{0}<j<-\frac{\log_{2}t}{2\beta}}
\sum\limits_{Q_{j,k}\subset Q_{r}}2^{qj(\gamma_{1}+\frac{n}{2})}\sum\limits_{|j-j'|\leq1}2^{-qj'(\gamma_{1}+\frac{n}{2})}\\
&\quad(M_{A}(f_{j'})(x))^{q}\chi(2^{j}x-k)\Big)^{1/q}\Big\|_{p}\\
\lesssim&|Q_{r}|^{\frac{\gamma_{2}}{n}-\frac{1}{p}}\sum\limits_{w\in\mathbb{Z}^{n}}(1+|w|)^{-N}\Big\|\Big(\sum\limits_{j_{0}<j<-\frac{\log_{2}t}{2\beta}}
\sum\limits_{Q_{j,k}\subset Q_{r}}\sum\limits_{|j-j'|\leq1}2^{q(j-j')(\gamma_{1}+\frac{n}{2})}\\
&\quad(M_{A}(f_{j'})(x))^{q}\chi(2^{j}x-k)\Big)^{1/q}\Big\|_{p}.
\end{array}$$
Because $|j-j'|\leq1$ and $j_{0}<j<-\log_{2}t/2\beta$, we have
$j'\geq-\log_{2}r-1$ and $j'\leq-\log_{2}t/2\beta+1$. On the other hand,
The facts $Q_{j,k}\subset Q_{r}$ and $Q_{j',k'}\subset Q^{w}_{j,k}$ imply that $Q_{j',k'}\subset Q^{w}_{r}$. We obtain that
$$\begin{array}{rl}
II^{\gamma_{1}, \gamma_{2}}_{p,q, Q_{r}}\lesssim&|Q_{r}|^{\frac{\gamma_{2}}{n}-\frac{1}{p}}\sum\limits_{w\in\mathbb{Z}^{n}}(1+|w|)^{-N}
\Big\|\Big(\sum\limits_{-\log_{2}r_{w}<j'<-\frac{\log_{2}t}{2\beta}+1}2^{qj'(\gamma_{1}+\frac{n}{2})}\\
&\quad\Big(\sum\limits_{(\varepsilon',k'): Q_{j',k'}\subset Q^{w}_{r}}|g^{\varepsilon'}_{j',k'}(t)|\chi(2^{j'}x-k')\Big)^{q}\Big)^{1/q}\Big\|_{p}\\
=:&M_{1}+M_{2},
\end{array}$$
where
$$\begin{array}{rl}
M_{1}=&|Q_{r}|^{\frac{\gamma_{2}}{n}-\frac{1}{p}}\sum\limits_{w\in\mathbb{Z}^{n}}(1+|w|)^{-N}\Big\|\Big(\sum\limits_{-\log_{2}r_{w}<j'<-\frac{\log_{2}t}{2\beta}}
2^{qj'(\gamma_{1}+\frac{n}{2})}\\
&\quad\Big(\sum\limits_{(\varepsilon',k'): Q_{j',k'}\subset Q^{w}_{r}}|g^{\varepsilon'}_{j',k'}(t)|\chi(2^{j'}x-k')\Big)^{q}\Big)^{1/q}\Big\|_{p}
\end{array}$$
and
$$\begin{array}{rl}
M_{2}&=|Q_{r}|^{\frac{\gamma_{2}}{n}-\frac{1}{p}}\sum\limits_{w\in\mathbb{Z}^{n}}(1+|w|)^{-N}\Big\|\Big(\sum\limits_{j'\geq\max\left\{-\log_{2}r_{w}, -\frac{\log_{2}t}{2\beta}\right\}}
2^{qj'(\gamma_{1}+\frac{n}{2})}\\
&\quad\Big(\sum\limits_{(\varepsilon',k'): Q_{j',k'}\subset Q^{w}_{r}}|g^{\varepsilon'}_{j',k'}(t)|\chi(2^{j'}x-k')\Big)^{q}\Big)^{1/q}\Big\|_{p}
\end{array}$$
Obviously $M_{1}\lesssim \|g\|_{\mathbb{F}^{\gamma_{1}, \gamma_{2}, II}_{p,q }}$. For $M_{1}$, if $j'\geq\max\{-\log_{2}r_{w}, -\frac{\log_{2}t}{2\beta}\}$,
then $(t2^{2j'\beta})^{mq}\geq1$ and
$$\begin{array}{rl}
M_{2}\lesssim&|Q_{r}|^{\frac{\gamma_{2}}{n}-\frac{1}{p}}\sum\limits_{w\in\mathbb{Z}^{n}}(1+|w|)^{-N}\Big\|\Big(\sum\limits_{j'\geq\max\left\{-\log_{2}r_{w}, -\frac{\log_{2}t}{2\beta}\right\}}
2^{qj'(\gamma_{1}+\frac{n}{2})}(t2^{2j'\beta})^{mq}\\
&\quad\Big(\sum\limits_{(\varepsilon',k'): Q_{j',k'}\subset Q^{w}_{r}}|g^{\varepsilon'}_{j',k'}(t)|\chi(2^{j'}x-k')\Big)^{q}\Big)^{1/q}\Big\|_{p}\\
\lesssim&|Q_{r}|^{\frac{\gamma_{2}}{n}-\frac{1}{p}}\sum\limits_{w\in\mathbb{Z}^{n}}t^{m}(1+|w|)^{-N}\Big\|\Big(\sum\limits_{j'\geq\max\left\{-\log_{2}r_{w}, -\frac{\log_{2}t}{2\beta}\right\}}
2^{qj'(\gamma_{1}+\frac{n}{2}+2m\beta)}\\
&\quad\sum\limits_{(\varepsilon',k'): Q_{j',k'}\subset Q^{w}_{r}}|g^{\varepsilon'}_{j',k'}(t)|^{q}\chi(2^{j'}x-k')\Big)^{1/q}\Big\|_{p}\\
\lesssim&\|g\|_{\mathbb{F}^{\gamma_{1}, \gamma_{2}, I}_{p,q,m}}
\end{array}$$

{\bf Step III:} $(R_{l}g)(t, x)\in\mathbb{F}^{\gamma_{1}, \gamma_{2}, III}_{p, q, m}$, $l=1, 2, \cdots, n$.

$$\begin{array}{rl}
III^{\gamma_{1},\gamma_{2}}_{p,q,m}\lesssim&|Q_{r}|^{\frac{\gamma_{2}}{n}-\frac{1}{p}}\Big\|\Big(\sum\limits_{(\varepsilon,j,k)\in\Lambda^{n}_{Q}}
2^{qj(\gamma_{1}+\frac{n}{2}+2m\beta)}
\int^{r^{2\beta}}_{2^{-2j\beta}}t^{mq}|a^{\varepsilon}_{j,k}(t)|^{q}\frac{dt}{t}\chi(2^{j}x-k)\Big)^{1/q}\Big\|_{p}\\
\lesssim&|Q_{r}|^{\frac{\gamma_{2}}{n}-\frac{1}{p}}\Big\|\Big(\sum\limits_{(\varepsilon,j,k)\in\Lambda^{n}_{Q}}
2^{qj(\gamma_{1}+\frac{n}{2}+2m\beta)}\\
&\quad\int^{r^{2\beta}}_{2^{-2j\beta}}t^{mq}\sum\limits_{|j-j'|\leq1}\Big(\sum\limits_{\varepsilon',k'}
\frac{|g^{\varepsilon'}_{j',k'}(t)|}{(1+|2^{j-j'}k'-k|)^{n+N_{0}}}\Big)^{q}\frac{dt}{t}\chi(2^{j}x-k)\Big)^{1/q}\Big\|_{p}\\
\lesssim&|Q_{r}|^{\frac{\gamma_{2}}{n}-\frac{1}{p}}\Big\|\Big(\sum\limits_{(\varepsilon,j,k)\in\Lambda^{n}_{Q}}
2^{qj(\gamma_{1}+\frac{n}{2}+2m\beta)}\\
&\quad\int^{r^{2\beta}}_{2^{-2j\beta}}t^{mq}\sum\limits_{|j-j'|\leq1}\sum\limits_{\varepsilon',k'}
\frac{|g^{\varepsilon'}_{j',k'}(t)|^{q}}{(1+|2^{j-j'}k'-k|)^{n+N_{0}}}\frac{dt}{t}\chi(2^{j}x-k)\Big)^{1/q}\Big\|_{p}\\
\lesssim&|Q_{r}|^{\frac{\gamma_{2}}{n}-\frac{1}{p}}\Big\|\Big(\sum\limits_{(\varepsilon,j,k)\in\Lambda^{n}_{Q}}
2^{qj(\gamma_{1}+\frac{n}{2}+2m\beta)}\sum\limits_{|j-j'|\leq1}\sum\limits_{\varepsilon',k'}
\frac{1}{(1+|2^{j-j'}k'-k|)^{n+N_{0}}}\\
&\quad\Big(\int^{r^{2\beta}}_{2^{-2j\beta}}t^{mq}|g^{\varepsilon'}_{j',k'}(t)|^{q}\frac{dt}{t}\Big)\chi(2^{j}x-k)\Big)^{1/q}\Big\|_{p}.
\end{array}$$
Let
$$\begin{array}{rl}
|b^{\varepsilon'}_{j',k'}|&=\int^{r^{2\beta}}_{2^{-2j\beta}}t^{qm}|g^{\varepsilon'}_{j',k'}(t)|^{q}\frac{dt}{t}.
\end{array}$$
Because $|j-j'|\leq1$,
$$\begin{array}{rl}
|b^{\varepsilon'}_{j',k'}|&\lesssim\int^{r^{2\beta}}_{2^{-2j'\beta-2\beta}}t^{qm}|g^{\varepsilon'}_{j',k'}(t)|^{q}\frac{dt}{t}.
\end{array}$$
We can get
$$\begin{array}{rl}
III^{\gamma_{1},\gamma_{2}}_{p,q,m}\lesssim&|Q_{r}|^{\frac{\gamma_{2}}{n}-\frac{1}{p}}\Big\|\Big(\sum\limits_{(\varepsilon,j,k)\in\Lambda^{n}_{Q}}
2^{qj(\gamma_{1}+\frac{n}{2}+2m\beta)}\sum\limits_{|j-j'|\leq1}\\
&\quad\sum\limits_{\varepsilon',k'}
\frac{|b^{\varepsilon'}_{j',k'}|}{(1+|2^{j-j'}k'-k|)^{n+N_{0}}}\chi(2^{j}x-k)\Big)^{1/q}\Big\|_{p}\\
\lesssim&|Q_{r}|^{\frac{\gamma_{2}}{n}-\frac{1}{p}}\Big\|\Big(\sum\limits_{(\varepsilon,j,k)\in\Lambda^{n}_{Q}}
2^{qj(\gamma_{1}+\frac{n}{2}+2m\beta)}\sum\limits_{|j-j'|\leq1}\\
&\quad\Big(\sum\limits_{\varepsilon',k'}
\frac{|b^{\varepsilon'}_{j',k'}|^{1/q}}{(1+|2^{j-j'}k'-k|)^{n+N_{0}}}\Big)^{q}\chi(2^{j}x-k)\Big)^{1/q}\Big\|_{p}
\end{array}$$
Let
$$\begin{array}{rl}
g_{j'}&=\sum\limits_{Q_{j',k'}\subset Q^{w}_{r}}2^{j'(\gamma_{1}+\frac{n}{2}+2m\beta)}|b^{\varepsilon'}_{j',k'}(t)|^{1/q}\chi(2^{j'}x-k').
\end{array}$$
Then
$$\begin{array}{rl}
III^{\gamma_{1},\gamma_{2}}_{p,q,m}\lesssim&|Q_{r}|^{\frac{\gamma_{2}}{n}-\frac{1}{p}}\sum\limits_{w\in\mathbb{Z}^{n}}\frac{1}{(1+|w|)^{N}}
\Big\|\Big(\sum\limits_{j'\geq-\log_{2}r_{w}}(M_{A}g_{j'}(x))^{q}\Big)^{1/q}\Big\|_{p}\\
\lesssim&|Q_{r}|^{\frac{\gamma_{2}}{n}-\frac{1}{p}}\sum\limits_{w\in\mathbb{Z}^{n}}\frac{1}{(1+|w|)^{N}}
\Big\|\Big(\sum\limits_{j'\geq-\log_{2}r_{w}}|g_{j'}(x)|^{q}\Big)^{1/q}\Big\|_{p}\\
\lesssim&|Q_{r}|^{\frac{\gamma_{2}}{n}-\frac{1}{p}}\sum\limits_{w\in\mathbb{Z}^{n}}\frac{1}{(1+|w|)^{N}}
\Big\|\Big(\sum\limits_{j'\geq-\log_{2}r_{w}}2^{qj'(\gamma_{1}+\frac{n}{2}+2m\beta)}\\
&\quad\Big(\sum\limits_{Q_{j',k'}\subset Q^{w}_{r}}|b^{\varepsilon'}_{j',k'}(t)|^{1/q}\chi(2^{j'}x-k')\Big)^{q}\Big)^{1/q}\Big\|_{p}.
\end{array}$$
For fixed $j'$, there exist only one $Q_{j',k'}$ such that $x\in Q_{j',k'}$. Then
$$\begin{array}{rl}
&\Big(\sum\limits_{Q_{j',k'}\subset Q^{w}_{r}}|b^{\varepsilon'}_{j',k'}(t)|^{1/q}\chi(2^{j'}x-k')\Big)^{q}\\
\lesssim&\sum\limits_{Q_{j',k'}\subset Q^{w}_{r}}|b^{\varepsilon'}_{j',k'}(t)|\chi(2^{j'}x-k').
\end{array}$$
Hence
$$\begin{array}{rl}
III^{\gamma_{1},\gamma_{2}}_{p,q,m}\lesssim&|Q_{r}|^{\frac{\gamma_{2}}{n}-\frac{1}{p}}\sum\limits_{w\in\mathbb{Z}^{n}}\frac{1}{(1+|w|)^{N}}
\Big\|\Big(\sum\limits_{j'\geq-\log_{2}r_{w}}2^{qj'(\gamma_{1}+\frac{n}{2}+2m\beta)}\\
&\quad\sum\limits_{Q_{j',k'}\subset Q^{w}_{r}}|b^{\varepsilon'}_{j',k'}(t)|\chi(2^{j'}x-k')\Big)^{1/q}\Big\|_{p}\\
\lesssim&|Q_{r}|^{\frac{\gamma_{2}}{n}-\frac{1}{p}}\sum\limits_{w\in\mathbb{Z}^{n}}\frac{1}{(1+|w|)^{N}}
\Big\|\Big(\sum\limits_{j'\geq-\log_{2}r_{w}}2^{qj'(\gamma_{1}+\frac{n}{2}+2m\beta)}\\
&\quad\sum\limits_{Q_{j',k'}\subset Q^{w}_{r}}\int^{r^{2\beta}}_{2^{-2j'\beta-2\beta}}t^{qm}|g^{\varepsilon'}_{j',k'}(t)|^{q}\frac{dt}{t}\chi(2^{j'}x-k')\Big)^{1/q}\Big\|_{p}\\
=:&M_{1}+M_{2},
\end{array}$$
where
$$\begin{array}{rl}
M_{1}=&|Q_{r}|^{\frac{\gamma_{2}}{n}-\frac{1}{p}}\sum\limits_{w\in\mathbb{Z}^{n}}\frac{1}{(1+|w|)^{N}}
\Big\|\Big(\sum\limits_{j'\geq-\log_{2}r_{w}}2^{qj'(\gamma_{1}+\frac{n}{2}+2m\beta)}\\
&\quad\sum\limits_{Q_{j',k'}\subset Q^{w}_{r}}\int^{r^{2\beta}}_{2^{-2j'\beta}}t^{qm}|g^{\varepsilon'}_{j',k'}(t)|^{q}\frac{dt}{t}\chi(2^{j'}x-k')\Big)^{1/q}\Big\|_{p}
\end{array}$$
and
$$\begin{array}{rl}
M_{2}=&|Q_{r}|^{\frac{\gamma_{2}}{n}-\frac{1}{p}}\sum\limits_{w\in\mathbb{Z}^{n}}\frac{1}{(1+|w|)^{N}}
\Big\|\Big(\sum\limits_{j'\geq-\log_{2}r_{w}}2^{qj'(\gamma_{1}+\frac{n}{2}+2m\beta)}\\
&\quad\sum\limits_{Q_{j',k'}\subset Q^{w}_{r}}\int^{2^{-2j'\beta}}_{2^{-2j'\beta-2\beta}}t^{qm}|g^{\varepsilon'}_{j',k'}(t)|^{q}\frac{dt}{t}\chi(2^{j'}x-k')\Big)^{1/q}\Big\|_{p}.
\end{array}$$
It is easy to see that
$M_{1}\lesssim\|g\|_{\mathbb{F}^{\gamma_{1},\gamma_{2}, III}_{p,q,m}}$. For $M_{2}$, because $t<2^{2j'\beta}$, $(t2^{2j'\beta})^{m}\lesssim (t2^{2j'\beta})^{m'}$. Then
$$\begin{array}{rl}
M_{2}\lesssim&|Q_{r}|^{\frac{\gamma_{2}}{n}-\frac{1}{p}}\sum\limits_{w\in\mathbb{Z}^{n}}\frac{1}{(1+|w|)^{N}}
\Big\|\Big(\sum\limits_{j'\geq-\log_{2}r_{w}}2^{qj'(\gamma_{1}+\frac{n}{2}+2m'\beta)}\\
&\quad\sum\limits_{Q_{j',k'}\subset Q^{w}_{r}}\int^{2^{-2j'\beta}}_{0}t^{qm'}|g^{\varepsilon'}_{j',k'}(t)|^{q}\frac{dt}{t}\chi(2^{j'}x-k')\Big)^{1/q}\Big\|_{p}\\
\lesssim&\|g\|_{\mathbb{F}^{\gamma_{1},\gamma_{2}, IV}_{p,q,m'}}.
\end{array}$$

{\bf Step IV:} $(R_{l}g)(t, x)\in\mathbb{F}^{\gamma_{1}, \gamma_{2}, IV}_{p, q, m'}$, $l=1, 2, \cdots, n$.
$$\begin{array}{rl}
IV^{\gamma_{1}, \gamma_{2}}_{p,q,m'}=&|Q_{r}|^{\frac{\gamma_{2}}{n}-\frac{1}{p}}\Big\|\Big(\sum\limits_{(\varepsilon,j,k)\in\Lambda^{n}_{Q}}
2^{qj(\gamma_{1}+\frac{n}{2}+2m'\beta)}\\
&\quad\int_{0}^{2^{-2j\beta}}t^{m'q}|a^{\varepsilon}_{j,k}(t)|^{q}\frac{dt}{t}\chi(2^{j}x-k)\Big)^{1/q}\Big\|_{p}\\
\lesssim&|Q_{r}|^{\frac{\gamma_{2}}{n}-\frac{1}{p}}\Big\|\Big(\sum\limits_{(\varepsilon,j,k)\in\Lambda^{n}_{Q}}
2^{qj(\gamma_{1}+\frac{n}{2}+2m'\beta)}\sum\limits_{|j-j'|\leq1}\\
&\quad\int_{0}^{2^{-2j\beta}}t^{m'q}\Big(\sum\limits_{\varepsilon',k'}
\frac{|g^{\varepsilon'}_{j',k'}(t)|}{(1+|2^{j-j'}k'-k|)^{n+N_{0}}}\Big)^{q}\frac{dt}{t}\chi(2^{j}x-k)\Big)^{1/q}\Big\|_{p}\\
\lesssim&|Q_{r}|^{\frac{\gamma_{2}}{n}-\frac{1}{p}}\Big\|\Big(\sum\limits_{(\varepsilon,j,k)\in\Lambda^{n}_{Q}}
2^{qj(\gamma_{1}+\frac{n}{2}+2m'\beta)}\sum\limits_{|j-j'|\leq1}\\
&\quad\int_{0}^{2^{-2j\beta}}t^{m'q}\sum\limits_{\varepsilon',k'}
\frac{|g^{\varepsilon'}_{j',k'}(t)|^{q}}{(1+|2^{j-j'}k'-k|)^{n+N_{0}}}\frac{dt}{t}\chi(2^{j}x-k)\Big)^{1/q}\Big\|_{p}.
\end{array}$$
because $0<t<2^{-2j\beta}$ and $|j-j'|\leq1$, $2^{-2j\beta}\leq2^{-2j'\beta+2\beta}$. We can get
$$\begin{array}{rl}
IV^{\gamma_{1}, \gamma_{2}}_{p,q,m'}\lesssim&|Q_{r}|^{\frac{\gamma_{2}}{n}-\frac{1}{p}}\Big\|\Big(\sum\limits_{(\varepsilon,j,k)\in\Lambda^{n}_{Q}}
2^{qj(\gamma_{1}+\frac{n}{2}+2m'\beta)}\sum\limits_{|j-j'|\leq1}\\
&\quad\sum\limits_{\varepsilon',k'}\Big(\int_{0}^{2^{-2j\beta}}t^{m'q}|g^{\varepsilon'}_{j',k'}(t)|^{q}\frac{dt}{t}\Big)
\frac{1}{(1+|2^{j-j'}k'-k|)^{n+N_{0}}}\chi(2^{j}x-k)\Big)^{1/q}\Big\|_{p}.
\end{array}$$
Let
$$\begin{array}{rl}
b^{\varepsilon'}_{j', k'}=\int_{0}^{2^{-2j\beta}}t^{m'q}|g^{\varepsilon'}_{j',k'}(t)|^{q}\frac{dt}{t}.
\end{array}$$
Then
$$\begin{array}{rl}
IV^{\gamma_{1}, \gamma_{2}}_{p,q,m'}\lesssim&|Q_{r}|^{\frac{\gamma_{2}}{n}-\frac{1}{p}}\Big\|\Big[\sum\limits_{(\varepsilon,j,k)\in\Lambda^{n}_{Q}}
2^{qj(\gamma_{1}+\frac{n}{2}+2m'\beta)}\sum\limits_{|j-j'|\leq1}\\
&\quad\sum\limits_{\varepsilon',k'}\frac{|b^{\varepsilon'}_{j', k'}|}{(1+|2^{j-j'}k'-k|)^{n+N_{0}}}\chi(2^{j}x-k)\Big]^{1/q}\Big\|_{p}\\
\lesssim&|Q_{r}|^{\frac{\gamma_{2}}{n}-\frac{1}{p}}\sum\limits_{w\in\mathbb{Z}^{n}}(1+|w|)^{-N}\Big\|\Big[\sum\limits_{(\varepsilon,j,k)\in\Lambda^{n}_{Q}}
2^{qj(\gamma_{1}+\frac{n}{2}+2m'\beta)}\sum\limits_{|j-j'|\leq1}\\
&\quad\Big(\sum\limits_{Q_{j',k'}\subset Q_{r}^{w}}\frac{|b^{\varepsilon'}_{j', k'}|^{1/q}}{(1+|2^{j-j'}k'-k|)^{n+N_{0}}}\Big)^{q}\chi(2^{j}x-k)\Big]^{1/q}\Big\|_{p}
\end{array}$$
Take $$\begin{array}{rl}
g_{j'}&=2^{j'(\gamma_{1}+\frac{n}{2}+2m'\beta)}\sum\limits_{Q_{j',k'}\subset Q^{w}_{r}}|b^{\varepsilon'}_{j',k'}|^{1/q}\chi(2^{j'}x-k').
\end{array}$$
We obtain
$$\begin{array}{rl}
IV^{\gamma_{1}, \gamma_{2}}_{p,q,m'}\lesssim&|Q_{r}|^{\frac{\gamma_{2}}{n}-\frac{1}{p}}\sum\limits_{w\in\mathbb{Z}^{n}}(1+|w|)^{-N}\Big\|\Big[\sum\limits_{j'\geq-\log_{2}r_{w}}
|M_{A}(g_{j'})(x)|^{q}\Big]^{1/q}\Big\|_{p}\\
\lesssim&|Q_{r}|^{\frac{\gamma_{2}}{n}-\frac{1}{p}}\sum\limits_{w\in\mathbb{Z}^{n}}(1+|w|)^{-N}\Big\|\Big[\sum\limits_{j'\geq-\log_{2}r_{w}}
|g_{j'}(x)|^{q}\Big]^{1/q}\Big\|_{p}\\
\lesssim&|Q_{r}|^{\frac{\gamma_{2}}{n}-\frac{1}{p}}\sum\limits_{w\in\mathbb{Z}^{n}}(1+|w|)^{-N}\Big\|\Big[\sum\limits_{j'\geq-\log_{2}r_{w}}
2^{qj'(\gamma_{1}+\frac{n}{2}+2m'\beta)}\\
&\quad\Big(\sum\limits_{Q_{j',k'}\subset Q^{w}_{r}}|b^{\varepsilon'}_{j',k'}|^{1/q}\chi(2^{j'}x-k')\Big)^{q}\Big]^{1/q}\Big\|_{p}\\
\lesssim&|Q_{r}|^{\frac{\gamma_{2}}{n}-\frac{1}{p}}\sum\limits_{w\in\mathbb{Z}^{n}}(1+|w|)^{-N}\Big\|\Big[\sum\limits_{j'\geq-\log_{2}r_{w}}
2^{qj'(\gamma_{1}+\frac{n}{2}+2m'\beta)}\\
&\quad\sum\limits_{Q_{j',k'}\subset Q^{w}_{r}}|b^{\varepsilon'}_{j',k'}|\chi(2^{j'}x-k')\Big]^{1/q}\Big\|_{p}\\
\lesssim&|Q_{r}|^{\frac{\gamma_{2}}{n}-\frac{1}{p}}\sum\limits_{w\in\mathbb{Z}^{n}}(1+|w|)^{-N}\Big\|\Big[\sum\limits_{j'\geq-\log_{2}r_{w}}
2^{qj'(\gamma_{1}+\frac{n}{2}+2m'\beta)}\\
&\quad\sum\limits_{Q_{j',k'}\subset Q^{w}_{r}}\Big(\int_{0}^{2^{-2j'\beta+2\beta}}t^{m'q}|g^{\varepsilon'}_{j',k'}(t)|^{q}\frac{dt}{t}\Big)\chi(2^{j'}x-k')\Big]^{1/q}\Big\|_{p}.
\end{array}$$
It is easy to see that
$$\begin{array}{rl}
IV^{\gamma_{1}, \gamma_{2}}_{p,q,m'}\lesssim&|Q_{r}|^{\frac{\gamma_{2}}{n}-\frac{1}{p}}\sum\limits_{w\in\mathbb{Z}^{n}}(1+|w|)^{-N}\Big\|\Big[\sum\limits_{j'\geq-\log_{2}r_{w}}
2^{qj'(\gamma_{1}+\frac{n}{2}+2m'\beta)}\\
&\quad\sum\limits_{Q_{j',k'}\subset Q^{w}_{r}}\Big(\int_{0}^{2^{-2j'\beta}}t^{m'q}|g^{\varepsilon'}_{j',k'}(t)|^{q}\frac{dt}{t}\Big)\chi(2^{j'}x-k')\Big]^{1/q}\Big\|_{p}\\
+&|Q_{r}|^{\frac{\gamma_{2}}{n}-\frac{1}{p}}\sum\limits_{w\in\mathbb{Z}^{n}}(1+|w|)^{-N}\Big\|\Big[\sum\limits_{j'\geq-\log_{2}r_{w}}
2^{qj'(\gamma_{1}+\frac{n}{2}+2m'\beta)}\\
&\quad\sum\limits_{Q_{j',k'}\subset Q^{w}_{r}}\Big(\int_{2^{-2j'\beta}}^{2^{-2j'\beta+2\beta}}t^{m'q}|g^{\varepsilon'}_{j',k'}(t)|^{q}\frac{dt}{t}\Big)\chi(2^{j'}x-k')\Big]^{1/q}\Big\|_{p}\\
\lesssim&\|g\|_{\mathbb{F}^{\gamma_{1},\gamma_{2}, III}_{p,q,m}}+\|g\|_{\mathbb{F}^{\gamma_{1},\gamma_{2}, IV}_{p,q,m'}}.
\end{array}$$

\end{proof}


\begin{thebibliography}{30}



\bibitem{Ad}
D. Adams, \textit{A note on Choquet integrals with respect to
Hausdorff capacity}, { Function Spaces and Applications (Lund,
1986)}, {Lecture Notes in Math.}\textbf{1302}, 115-24, Springer,
Berlin, 1988.

\bibitem{Al}
J. Alvarez, \textit{Continuity of Calder\'on-Zygmund type operators
on the predual of a Morrey space,} { Clifford algebras in analysis
and related topics (Fayetteville, AR, 1993)}, 309-319, Stud. Adv.
Math., CRC, Boca Raton, FL, 1996.



\bibitem{ASX}R. Aulaskari, D. Stegenga, J. Xiao, \textit{Some subclasses of BMOA
and their characterization in terms of} \text{Carleson measure}, {
Rocky Mountain J. Math.} \textbf{26}(2) (1996), 485-506.



\bibitem{AXZ}
R. Aulaskari, J. Xiao, R. Zhao, \textit{On subspaces and subsets
of BMOA and UBC}, {Analysis} \textbf{15} (1995), 101-121.

\bibitem{C} M. Cannone, \textit{Ondelettes, paraproduits et Navier-Stokes}, {Diderot Editeur}, Paris, 1995.



\bibitem{C2}
M. Cannone, \textit{Harmonic analysis tools for solving the
incompressible Navier-Stokes equations}, { Handbook of mathematical
fluid dynamics, Vol. III}, 161-244, North-Holland, Amsterdam, 2004.










\bibitem{CMS} R. Coifman, Y. Meyer, E. M. Stein, \textit{Some new function spaces and their applications to harmonic
analysis,} {{J. Funct. Anal.} \textbf{62}} (1985), 304-335.



\bibitem{DX}
G. Dafni, J. Xiao, \textit{Some new tent spaces and duality
theorem for fractional Carleson measures and
$Q_{\alpha}(\mathbb{R}^{2})$}, {J. Funct. Anal.} \textbf{208}
(2004), 377-422.





\bibitem{EJPX}
M. Ess\'en, S. Janson, L. Peng,  J. Xiao, \textit{$Q$ spaces of
several real variables}, {Indiana Univ. Math. J.} \textbf{49}
(2000), 575-615.

\bibitem{FJW}
M. Frazier, B. Jawerth,  G. Weiss, \textit{Littlewood-Paley
Theory and the Study of Function Spaces}, CBMS Reg. Conf. Ser. Math.
vol. \textbf{79}, Amer. Math. Soc., Providence, RI, 1991.




\bibitem{KT}
H. Koch, D. Tataru, \textit{Well-posedness for the Navier-Stokes
equations}, {Adv. Math.} \textbf{157} (2001), 22-35.




\bibitem{Lem}
P. G. Lemari\'e-Rieusset, \textit{Recent Development in the
Navier-Stokes Problem}, Chapman \& Hall/CRC Press, Boca Raton, 2002.

\bibitem{LXY} P. Li, J. Xiao, Q. Yang, \textit{Global Mild Solutions of Fractional Naiver-Stokes Equations with Small Initial Data in
Critical Besov-Q Spaces}, {arXiv:1212.0766v4 [math.AP] }.

\bibitem{LZ1}
P. Li, Z. Zhai, \textit{Well-posedness and regularity of
generalized Navier-Stokes equations in some critical Q-spaces,} {J.
Funct. Anal.} \textbf{259} (2010), 2457-2519

\bibitem{LZ2}
P. Li, Z. Zhai, \textit{Riesz transforms on $Q$-type spaces with
application to quasi-geostrophic equation}, {Taiwanese Journal of
Mathematics}, to appear.

\bibitem{LSUYY} Y. Liang, Y. Sawano, T. Ullrich, D. Yang and W. Yuan, \textit{New characterizations of Besov-Triebel-Lizorkin-Hausdorff spaces including coorbits and wavelets}, \textbf{18} (2012), 1067-1111.



\bibitem{LinYang}
C. Lin, Q. Yang, \textit{Semigroup characterization of Besov type Morrey spaces
and well-posedness of generalized Navier-Stokes equations},
Journal of Differential Equations, 254(2013) 804-846


\bibitem{Me}
Y. Meyer, \textit{Ondelettes et op\'erateurs, I et II,} Hermann,
Paris, 1991-1992.

\bibitem{MY}
Y. Meyer, Q. Yang, \textit{Continuity of Calder\'on-Zygmund
operators on Besov or Triebel-Lizorkin spaces,} {Anal. Appl.}
(Singap.) {\bf 6} (2008), 51-81.





\bibitem{C. Miao B. Yuan  B. Zhang}
C. Miao, B. Yuan, B. Zhang,
 \textit{Well-posedness of the Cauchy problem for the fractional power dissipative equations,}
 {Nonlinear Anal. TMA} \textbf{68} (2008), 461-484.




\bibitem{Mo}
C. Morrey, \textit{On the solutions of quasi-linear elliptic partial
differential equations}, Trans Amer. Math. Soc. {\bf 43}, (1938),
126-166.





\bibitem{NX}A. Nicolau, J. Xiao, \textit{Bounded functions in M$\ddot{o}$bius invariant Dirichlet
spaces,} { J. Funct. Anal.} \textbf{150} (1997), 383-425.





\bibitem{P}
J. Peetre, \textit{New Thoughts on Besov Spaces}, Duke Univ. Math.
Ser., Duke Univ. Press, Durham, 1976.

\bibitem{PY} L. Peng, Q. Yang, \textit{Predual spaces for $Q$ spaces},
Acta Math. Sci. Ser. B Engl. Ed. {\bf 29} (2009), 243-250.







\bibitem{Sawano}
Y. Sawano, \textit{Wavelet characterization of Besov-Morrey and
Triebel-Lizorkin-Morrey spaces}, {Funct. Approx. Comment. Math.}
{\bf 38} (2008), 93-107.

\bibitem{Triebel-1}H. Triebel, \textit{Theory of Function Spaces}, Monogr. Math., vol. 78, Birkh\"auser, Basel, 1983.

\bibitem{Triebel-2}H. Triebel, \textit{Theory of Function Spaces II}, Monogr. Math., vol. 84, Birkh\"auser, Basel, 1992.

\bibitem{Woj} P. Wojtaszczyk, \textit{A Mathematical Introduction to Wavelets},
London Mathematical Society Student Texts {\bf 37}, Cambridge University Press, Cambridge, 1997.


\bibitem{WX}
Z. Wu, C. Xie, \textit{$Q$ spaces and Morrey spaces}, J. Func.
Anal. {\bf 201} (2003), 282-297.





\bibitem{X1}
J. Xiao, \textit{Homothetic variant of fractional Sobolev space with
application to Navier-Stokes system}, Dyn. Partial Differ. Equ. {\bf
4} (2007), 227-245.




\bibitem{X2} J. Xiao, \textit{Holomorphic Q Class}, { Lecture Notes in
Math.} \textbf{}1767, {Springer}, Berlin, 2001.

\bibitem{YY} D. Yang, W. Yuan, \textit{A new class of function spaces connecting Triebel-Lizorkin spaces and $Q$
spaces}, {J. Funct. Anal.} \textbf{255} (2008), 2760-2809.





\bibitem{Yang1}
Q. Yang, \textit{Wavelet and Distribution}, Beijing Science and
Technology Press, Beijing, 2002.



\bibitem{YSY} W. Yuan, W. Sickel, D. Yang, \textit{Morrey and Campanato Meet Besov, Lizorkin and Triebel},
Lecture Notes in Mathematics 2005, Editors: J. M. Morel, Cachan F. Takens, Groningen B. Teissier, Paris.
Springer Heidelberg Dordrecht London New York, 2010.
\end{thebibliography}
\end{document}